\documentclass[12pt,reqno]{amsart}

\usepackage{amsmath,amsfonts,amsthm,graphicx,float,color}
\usepackage{cite}
\usepackage[hidelinks]{hyperref}
\usepackage{xcolor}
\hypersetup{
    colorlinks,
    linkcolor={black},
    citecolor={black},
    urlcolor={black}
}
\newcommand\blfootnote[1]{%
  \begingroup
  \renewcommand\thefootnote{}\footnote{#1}%
  \addtocounter{footnote}{-1}%
  \endgroup
}
\usepackage{cancel}
\usepackage[margin = 1in]{geometry}
\usepackage{mathrsfs}
\usepackage{mathabx}
\usepackage{dsfont}

\newcommand{\R}{\mathbb{R}}

\newcommand{\C}{\mathbb{C}}

\newcommand{\ve}{\varepsilon}

\newcommand{\lp}{\left(}
\newcommand{\rp}{\right)}

\newcommand{\wt}{\widetilde}
\newcommand{\wh}{\widehat}

\DeclareMathOperator{\Id}{I}

\newcommand{\supp}{\textnormal{supp\,}}
\newcommand{\inj}{\textnormal{inj}}
\newcommand{\D}{\mathcal{D}}

\newtheorem{thm}{Theorem}
\newtheorem{lem}{Lemma}
\newtheorem{prop}[lem]{Proposition}
\newtheorem{cor}[lem]{Corollary}
\newtheorem{rmk}[lem]{Remark}

\newcommand{\thmref}[1]{Theorem~\ref{#1}}

\newcommand{\lemref}[1]{Lemma~\ref{#1}}

\newcommand{\propref}[1]{Proposition~\ref{#1}}

\definecolor{bpurple}{RGB}{170,0,200}
\definecolor{bgreen}{RGB}{60,150,90}
\definecolor{bblue}{RGB}{50,70,200}

\numberwithin{equation}{section}
\numberwithin{lem}{section}
\begin{document}

\begin{abstract}
\footnotesize
In this paper, we study the two-point Weyl Law for the Laplace-Beltrami operator on a smooth, compact Riemannian manifold $M$ with no conjugate points. That is, we find the asymptotic behavior of the Schwartz kernel, $E_\lambda(x,y)$, of the projection operator from $L^2(M)$ onto the direct sum of eigenspaces with eigenvalue smaller than $\lambda^2$ as $\lambda \to\infty$. In the regime where $x,y$ are restricted to a compact neighborhood of the diagonal in $M\times M$, we obtain a uniform logarithmic improvement in the remainder of the asymptotic expansion for $E_\lambda$ and its derivatives of all orders, which generalizes a result of B\'erard, who treated the on-diagonal case $E_\lambda(x,x)$. When $x,y$ avoid a compact neighborhood of the diagonal, we obtain this same improvement in an upper bound for $E_\lambda$. Our results imply that the rescaled covariance kernel of a monochromatic random wave locally converges in the $C^\infty$ topology to a universal scaling limit at an inverse logarithmic rate.
\end{abstract}

\title[A logarithmic improvement on manifolds without conjugate points]{A logarithmic improvement in the two-point Weyl Law for manifolds without conjugate points}
\author{Blake Keeler}
\date{\today}
\maketitle

\nocite{Bonthonneau2017,Berard1977,Canzani2015,Canzani2018,ColindeVerdiere1973,DNPR2020,DuistermaatGuillemin1975,Hormander1968,HormanderBook1983,HormanderBook1985,SoggeBook2014,SoggeBook2017,SoggeZelditch2002,TrevesBook1975}







\blfootnote{To Appear in Annales de l'institut Fourier}

\small
\section{Introduction}
\label{introduction}
Let $(M,g)$ be a smooth, compact Riemannian manifold without boundary, and denote by $\Delta_g$ its positive definite Laplace-Beltrami operator. Let $\{\varphi_j\}_{j=0}^\infty$ be an orthonormal basis of $L^2(M)$ consisting of eigenfunctions of $\Delta_g$ with
\[
\Delta_g\varphi_j = \lambda_j^2\varphi_j,\hskip0.2in \|\varphi_j\|_{L^2(M)} = 1,
\]
 where $0=\lambda_0 < \lambda_1 \le \lambda_2 \le \dotsm$ are repeated according to multiplicity. We may, without loss of generality, take the $\varphi_j$ to be real-valued. We are interested in the Schwartz kernel of the spectral projection operator
\[
E_\lambda: L^2(M) \to \bigoplus\limits_{\lambda_j\le \lambda}\ker(\Delta_g - \lambda_j^2),
\]
which, in the above basis, takes the form
\[
E_\lambda(x,y) = \sum\limits_{\lambda_j\le \lambda}\varphi_j(x)\varphi_j(y)
\]
on $M\times M.$ This kernel is called the \textit{spectral function} of $\Delta_g.$ In this article, we investigate the two-point Weyl law for the spectral function, i.e. the asymptotic behavior of $E_\lambda(x,y)$ in the high-frequency limit $\lambda \to \infty.$ In the general case, the ``near-diagonal" behavior of $E_\lambda$ is known to be given by
\begin{equation}
\label{standardweyl}
E_\lambda(x,y) = \frac{\lambda^{n}}{(2\pi)^{n}}\int\limits_{B_x^*M} e^{i\lambda\langle\exp_x^{-1}(y),\xi\rangle_{g}}\frac{d\xi}{\sqrt{\det g_x}} + R_\lambda(x,y),
\end{equation}
where $B_x^*M$ is the unit ball in the cotangent space at $x$, and for any multi-indices $\alpha,\beta$,
\begin{equation}\label{weyl_remainder}
\sup\limits_{d_g(x,y) \le \ve} |\partial_x^\alpha\partial_y^\beta R_\lambda(x,y)|  = \mathcal O(\lambda^{n-1+|\alpha|+|\beta|}),
\end{equation}
as $\lambda\to\infty$ for some $\ve > 0$ sufficiently small. Here $d_g$ is the Riemannian distance function, $\exp_x^{-1}$ is the inverse of the exponential map defined on a sufficiently small neighborhood of $x$, and $g_x$ denotes the metric at $x$. We remark that for the purposes of this formula, we regard $\exp_x^{-1}(y)$ and $\xi$ as elements of $T_x^*M$, rather than $T_xM$ to be consistent with standard conventions in the literature. Throughout this article we will always interpret norms and inner products with the subscript $g$ as operations using the co-metric on $T^*M$, unless otherwise stated. 

A more general version of the above asymptotic was proved for the spectral functions of arbitrary positive elliptic pseudodifferential operators by H\"ormander in \cite{Hormander1968}, generalizing earlier results of Avakumovic \cite{Avakumovic1956} and Levitan \cite{Levitan1953,Levitan1955} for the on-diagonal behavior in the case of the Laplacian. We also remark that the original result was not stated to include derivatives of the remainder function, but as mentioned in \cite{Canzani2018}, \eqref{weyl_remainder} follows directly from the wave kernel method (e.g. \cite[\S4]{SoggeBook2017}, \cite{Xu2004}). Complementary to the near-diagonal result of H\"ormander, an estimate on $E_\lambda$ when $x$ and $y$ are ``far apart" was obtained by Safarov \cite{Safarov1988}, who showed that if $K$ is any compact set in $M\times M$ which does not intersect the diagonal with the property that if $x,y\in K$, then $x$ and $y$ are not mutually focal and at least one of $x$ or $y$ is not a focal point, then 
\begin{equation}\label{safarov_estimate}
\sup\limits_{x,y\in K}|E_\lambda(x,y)|= o(\lambda^n)
\end{equation}
as $\lambda \to\infty$. Safarov and Vassiliev also obtained some results on the precise form of the second term in the on-diagonal Weyl law, and we direct the reader to \cite{SafarovVassilievBook1997} for more information. In this article, we present improvements in both \eqref{weyl_remainder} and \eqref{safarov_estimate}, under the assumption that $(M,g)$ has no conjugate points. In the fully generic case, it is known that \eqref{weyl_remainder} is sharp, and this is easily shown by considering the zonal harmonics on the round sphere $\mathbb S^{n-1}$ centered at $x$ and restricting to $E_\lambda(x,x)$. However, by making assumptions about the behavior of the geodesic flow, one can often obtain improvements in the remainder estimate \eqref{weyl_remainder}. For example, Canzani and Hanin showed that if one assumes that $x_0\in M$ is non-self focal, i.e. the loopset given by $\{\xi\in S_{x_0}^*M :\, \exp_{x_0}(t\xi) = x_0\text{ for some }t > 0\}$ has Liouville measure zero in the co-sphere fiber $S_{x_0}^*M$, then one can locally improve \eqref{weyl_remainder} to 
\[\sup\limits_{x,y\in B(x_0,r_\lambda)}\left|\partial_x^\alpha\partial_y^\beta R_\lambda(x,y)\right| = o(\lambda^{n-1+|\alpha|+|\beta|})\]
as $\lambda\to\infty$, where $\lambda\mapsto r_\lambda$ is a real-valued function with $r_\lambda = o(1)$ as $\lambda \to\infty$, and $B(x_0,r_\lambda)$ is the geodesic ball of radius $r_\lambda$ centered at $x_0$ \cite{Canzani2015,Canzani2018}. This result was an extension of the work of Safarov \cite{Safarov1988}, who proved a pointwise $o(\lambda^{n-1})$ estimate for the on-diagonal remainder $R_\lambda(x,x)$ without derivatives. The same on-diagonal result was later proved independently by Sogge and Zelditch with an alternative proof \cite{SoggeZelditch2002}. This on-diagonal estimate was itself a generalization of the Duistermaat-Guillemin Theorem for the eigenvalue counting function \cite{DuistermaatGuillemin1975,Ivrii1980}. A more quantitative improvement in the Weyl law was obtained by B\'erard \cite{Berard1977}, who showed that under the stronger assumption of nonpositive curvature, one can obtain a factor of $\frac{1}{\log\lambda}$ in \eqref{weyl_remainder} when $x = y$ and $|\alpha| = |\beta| = 0$. This result was extended by Bonthonneau \cite{Bonthonneau2017} to apply to the case where $(M,g)$ has no conjugate points, and this was accomplished by proving that certain technical geometric estimates required in \cite{Berard1977} still hold in this more general setting. In this article, we generalize this logarithmic improvement by showing that it also holds in the more delicate off-diagonal case. We also show that adding derivatives in $x,y$ yields the expected change in the remainder bound, which enables us to obtain a quantitative rate of convergence for the rescaled covariance kernels of monochromatic random waves in the $C^\infty$ topology. This is the content of our main theorem, stated below.
\newpage
\begin{thm}\label{logthm}
Let $(M,g)$ be a smooth, compact Riemannian manifold without boundary, of dimension $n \ge 2$. Suppose that $(M,g)$ has no conjugate points. Then, for any multiindices $\alpha,\beta$, there exist positive constants $C_{\alpha,\beta}$ and $\lambda_0$ such that the remainder in the asymptotic expansion \eqref{standardweyl} satisfies 
\[
\sup\limits_{d_g(x,y)\le \frac{1}{2}\inj(M,g)} \left| \partial_x^\alpha\partial_y^\beta R_\lambda(x,y) \right| \le \frac{C_{\alpha,\beta}\lambda^{n-1+|\alpha|+|\beta|}}{\log\lambda}.
\]
for all $\lambda\ge \lambda_0$. 
\end{thm}

An outline of the proof of \thmref{logthm} is given in Subsection \ref{outline}. By modifying the proof slightly, we also obtain an improved upper bound on derivatives of $E_\lambda$ itself when $x,y$ are bounded away from each other, in analogy to Safarov's estimate \eqref{safarov_estimate} from \cite{Safarov1988}.

\begin{thm}\label{logthm_2}
For $(M,g)$ as in \thmref{logthm} and any $\ve >0$, there exist constants $C_{\alpha,\beta,\ve},\lambda_0 > 0$ such that
\begin{equation}\label{E_lambda_bound}
\sup\limits_{d_g(x,y) \ge \ve} \left|\partial_x^\alpha\partial_y^\beta E_\lambda(x,y)\right|\le \frac{C_{\alpha,\beta,\ve}\lambda^{n-1 + |\alpha|+|\beta|}}{\log\lambda}
\end{equation}
for all $\lambda\ge \lambda_0$.
\end{thm}
The proof of \thmref{logthm_2} is largely contained within that of \thmref{logthm}, and the necessary modifications are discussed in Remark \ref{safarov_remark}.

A straightforward consequence of \thmref{logthm} is an asymptotic for the spectral cluster kernels defined by
\[
E_{_{\!(\lambda,\lambda+1]}}(x,y) = \sum\limits_{\lambda_j\in(\lambda,\lambda+1]}\varphi_j(x)\varphi_j(y),
\]
for $x,y\in M$.
In Section \ref{cluster_thm_proof}, we show that using polar coordinates and the fact that
\[
\int\limits_{\mathbb S^{n-1}}e^{i\langle w,\sigma\rangle}\,d\sigma = (2\pi)^{\frac{n}{2}}\frac{J_{\frac{n-2}{2}}(|w|)}{|w|^{\frac{n-2}{2}}},
\]
where $J_\nu$ denotes the Bessel function of the first kind of order $\nu$ and $d\sigma$ is the standard surface measure on $\mathbb S^{n-1}$, one obtains the following consequence.
 
\begin{thm}\label{spectral_cluster_thm}
For $(M,g)$ as in Theorem \ref{logthm} and for any multi-indices $\alpha,\beta$, there exist constants $C_{\alpha,\beta},\lambda_0 > 0$ such that for any $x,y\in M$ with $d_g(x,y)\le \frac{1}{2}\inj(M,g)$,
\[
\left|\partial_x^\alpha\partial_y^\beta\lp E_{_{\!(\lambda,\lambda+1]}}(x,y) - \frac{\lambda^{n-1}}{(2\pi)^\frac{n}{2}}\frac{J_{\frac{n-2}{2}}(\lambda d_g(x,y))}{(\lambda d_g(x,y))^{\frac{n-2}{2}}} \rp\right| \le \frac{C_{\alpha,\beta}\lambda^{n-1+|\alpha|+|\beta|}}{\log\lambda}
\]
whenever $\lambda \ge \lambda_0.$
\end{thm}

We note that Theorem \ref{spectral_cluster_thm} only gives the leading order behavior of $E_{_{\!(\lambda,\lambda+1]}}(x,y)$ when $d_g(x,y)$ is very small relative to $\frac{1}{\lambda}$. To illustrate this, let us take the case where $|\alpha| = |\beta| = 0$. By standard properties of Bessel functions, we have that
\[\left|\lambda^{n-1}\frac{J_{\frac{n-2}{2}}(\lambda d_g(x,y))}{(\lambda d_g(x,y))^{\frac{n-2}{2}}}\right|\le C\lambda^{n-1}(1 + \lambda d_g(x,y))^{-\frac{n-1}{2}}.\]
Hence, if $d_g(x,y) \ge \frac{(\log\lambda)^{\frac{2}{n-1}}}{\lambda}$, then 
\[\lambda^{n-1}(1 + \lambda d_g(x,y))^{-\frac{n-1}{2}} \le \lambda^{n-1}\lp1 + (\log\lambda)^\frac{2}{n-1}\rp^{-\frac{n-1}{2}} = \mathcal O\lp\frac{\lambda^{n-1}}{\log\lambda}\rp.\]
Thus, if $d_g(x,y)$ is too large relative to $\frac{1}{\lambda}$, \thmref{spectral_cluster_thm} simply gives the same upper bound on $E_{_{\!(\lambda,\lambda+1]}}(x,y)$ that one would obtain by applying \thmref{logthm_2} and Cauchy-Schwarz. A similar argument shows that \thmref{logthm} only gives the leading behavior when $d_g(x,y)$ is smaller than $\mathcal O\lp\lambda^{\frac{2}{n-1}-1}(\log\lambda)^{\frac{2}{n-1}}\rp.$

Off-diagonal cluster estimates such as Theorem \ref{spectral_cluster_thm} have applications in the study of monochromatic random waves, which are random fields of the form
\[
\psi_\lambda(x) = \lambda^{\frac{1-n}{2}}\hspace{-1.2em}\sum\limits_{\lambda_j\in(\lambda,\lambda+1]}\!\!\!\!a_j\varphi_j(x),
\]
for $x\in M,$ where the $a_j$ are i.i.d. standard Gaussian random variables with mean 0 and variance 1. Random waves of this form were first introduced on Riemannian manifolds in \cite{Zelditch2009} by Zelditch, who was motivated by Berry's conjecture, which suggests that on manifolds with chaotic dynamics, high-frequency eigenfunctions should behave like certain stationary Gaussian fields in Euclidean space (c.f. \cite{Berry1977,IngremeauRivera2020}). 

By the Kolmogorov extension theorem, the statistics of monochromatic random waves are completely characterized by their covariance kernels, or two-point correlation functions, which can be computed directly as
\[
\text{Cov}(\psi_\lambda(x),\psi_\lambda(y)) = \lambda^{1-n}E_{_{\!(\lambda,\lambda+1]}}(x,y).
\]
for $x,y\in M.$ Theorem \ref{spectral_cluster_thm} implies that for any $x_0\in M$, we have the following convergence result for the covariance kernel in rescaled normal coordinates.
 
\begin{cor}\label{normal_coords} Let $(M,g)$ be as in Theorem \ref{logthm}, fix $x_0\in M$, and let $\lambda\mapsto r_\lambda$ be a real-valued function such that $ r_\lambda =\mathcal O\lp\sqrt{\frac{\lambda}{\log\lambda}}\rp$ as $\lambda \to\infty$. Then, for all $\alpha,\beta$,
\[
\textnormal{Cov}\lp\psi_\lambda\lp\exp_{x_0}(\tfrac{u}{\lambda})\rp\!,\!\psi_\lambda\lp\exp_{x_0}(\tfrac{v}{\lambda})\rp\rp = \frac{J_{\frac{n-2}{2}}(|u-v|)}{(2\pi)^{\frac{n}{2}}|u-v|^{\frac{n-2}{2}}} + R(u,v,\lambda),
\]
where
\[\sup\limits_{|u|,|v|\le r_\lambda}|\partial_u^\alpha\partial_v^\beta R(u,v,\lambda)| = \mathcal O\lp\frac{1}{\log\lambda}\rp,\]
as $\lambda\to\infty,$ and we consider $u,v$ as elements of $\R^n\cong T_{x_0}^*M$ when taking the supremum. 

\end{cor}
 
 Here the implicit constant depends on the choices of $x_0$ and $ r_\lambda$, and on the order of differentiation. Note that although the radius $ r_\lambda$  gives a growing ball in the $u,v$ coordinates, this corresponds to a shrinking ball of radius $\frac{ r_\lambda}{\lambda} = \mathcal O\lp\frac{1}{\sqrt{\lambda\log\lambda}}\rp$ on $M$, and, as $\lambda \to \infty$, this is indeed smaller than $\frac{1}{2}\inj(M,g)$ as required by \thmref{spectral_cluster_thm}. One can prove this corllary by Taylor expanding the function $F(\tau) = \frac{J_\nu(\tau)}{\tau^\nu}$, with $\nu = \frac{n-2}{2}$, around $\tau = 0$ and using that $d_g(x,y) - \frac{|u-v|}{\lambda} = \mathcal O\lp\frac{|u-v|^2}{\lambda^2}\rp.$
  Here, $x = \exp_{x_0}(u/\lambda)$ and $y = \exp_{x_0}(v/\lambda)$. In doing this Taylor expansion, we find that if $|u-v|^2 \le\mathcal O\lp\frac{\lambda}{\log\lambda}\rp$, then the error is smaller than the proposed $\mathcal O\lp\frac{1}{\log\lambda}\rp$ bound, which determines our condition on $ r_\lambda,$ although we do not claim that this is the largest possible radius for which the result holds. Corollary \ref{normal_coords} shows that the rescaled covariance kernel of a monochromatic random wave locally converges to that of a Euclidean random wave of frequency 1 at a rate of $\frac{1}{\log\lambda}$ in the $C^\infty$-topology, and hence the limit is universal in that it depends only on the dimension $n$, not on $M$ itself. As an interesting application, we note that a recent work of Dierickx, Nourdin, Peccati, and Rossi utilizes the quantitative rate of convergence given in Corollary \ref{normal_coords} in the proof of a small-scale central limit theorem for the nodal lengths of monochromatic random waves on surfaces without conjugate points \cite[Theorem 1.5]{DNPR2020}.

 Under the assumption that $x_0$ is a non self-focal point, Canzani and Hanin proved $o(1)$ convergence in the $C^0$-topology in \cite{Canzani2015}, and then in the $C^\infty$ topology in \cite{Canzani2018}. However, without any further restrictions on the geometry, they were unable to obtain an explicit rate of convergence as $\lambda \to \infty.$ Our $\frac{1}{\log\lambda}$ estimate is a first step toward obtaining quantitative asymptotic improvements on the statistics of monochromatic random waves in the fairly generic setting of manifolds without conjugate points.

\subsection{Outline of the Proof of Theorem \ref{logthm}}\label{outline}
We first relate the spectral function $E_\lambda(x,y)$ to the Schwartz kernel $K(t,x,y)$ of the wave operator $\cos (t\sqrt\Delta_g)$ using the Fourier transform taking $\lambda\mapsto t$, along with an on-diagonal spectral cluster estimate. We are able to use on-diagonal results here because we only need upper bounds on the spectral clusters in this piece of the argument. This is done in Section \ref{spectralfunction}, although the proof of the relevant spectral cluster estimate is postponed to Appendix \ref{cluster_appendix}, since the proof technique is largely a repetition of arguments from Section \ref{asymptotics}.

The second step is to approximate $K(t,x,y)$ using the Hadamard parametrix, which is done in Section \ref{hadamard}. The fact that $(M,g)$ has no conjugate points allows us to lift to the universal cover $(\wt M,\wt g)$, which is diffeomorphic to $\R^n$ by the Cartan-Hadamard theorem. We induce a parametrix on the base manifold by projecting, i.e. by summing over the deck transformation group $\Gamma$, which results in an expansion of the form 
\begin{equation}\label{parametrix_general_form}
K(t,x,y) = \sum\limits_{\nu = 0}^\infty \sum\limits_{\gamma\in \Gamma} F_\nu(t,\wt x,\gamma\wt y) \hskip 0.2in \mod C^\infty,
\end{equation}
where $\wt x,\wt y$ are some chosen lifts of $x,y$, and where each $F_\nu$ is the product of a $C^\infty$ function and a homogeneous distribution of order $2\nu - n$. We do not reproduce the construction of the parametrix, since it has been done in great detail in other sources (e.g. \cite{Berard1977,HormanderBook1985,SoggeBook2014}). Instead we focus on identifying the structure of the distributions which comprise the parametrix and on proving that the error introduced by approximating $K(t,x,y)$ by a partial sum in \eqref{parametrix_general_form} is sufficiently small. 

Once we have reduced the proof of Theorem \ref{logthm} to estimating an integral involving the parametrix, we perform some explicit asymptotic analysis on the individual terms as $\lambda \to \infty$. This is the content of Section \ref{asymptotics}. It is here that our techniques make the most significant departure from the work of B\'erard \cite{Berard1977}, where $R_\lambda(x,x)$ is estimated. In \cite{Berard1977}, the leading order behavior is obtained from the term in the parametrix corresponding to $\gamma = \Id$, and so $d_{\wt g}(\wt x,\wt x) = 0$. This reduces the relevant oscillatory integrals to a very simple form. In our case, a notable difficulty is that $d_{\wt g}(\wt x, \wt y)$ may be quite small, but need not be exactly zero, and so the corresponding singularities of the parametrix at $t = \pm d_{\wt g}(\wt x,\wt y)$ are very close together, but do not necessarily coincide. We still obtain the leading order behavior when $\wt x$ and $\wt y$ are the closest possible lifts of $x,y$, which we may assume occurs when $\gamma = \Id,$ but we do not get the same simplifications as in \cite{Berard1977} if the distance between them is nonzero.  This requires us to use a very different formulation of the parametrix terms $F_\nu$, so that we can track the dependence on this distance, which yields a more complicated linear combination of oscillatory integrals to estimate. We obtain somewhat weaker control on these terms, but the bounds are all smaller than the claimed estimate in \thmref{logthm}, and so the final result still holds. For the case where $\gamma \ne \Id$, our proof hinges on the fact that $d_{\wt g}(\wt x,\gamma\wt y)$ is bounded uniformly away from zero, thus allowing for improved estimates from applying stationary phase. 

\subsection{Organization of the Paper}
Sections \ref{spectralfunction}, \ref{hadamard}, and \ref{asymptotics} are devoted to the proof of Theorem~\ref{logthm}. \thmref{logthm_2} follows from the same techniques, as discussed in Remark \ref{safarov_remark}. Then, in Section \ref{cluster_thm_proof}, we prove that Theorem \ref{logthm} implies Theorem~\ref{spectral_cluster_thm}. 

 Appendix \ref{appendix} contains an estimate on summations involving factors which localize the summand to a $\lambda$-dependent region. This estimate is used in the proof of \propref{exact_kernel_remainder}, but the method of its proof is not particularly instructive, and so we relegate it to an Appendix. Appendix \ref{geometric_estimates} contains the proofs of some technical differential geometry results regarding quantities appearing in the construction of the parametrix, which are essential for including derivatives in the main result. We rely heavily on Jacobi field techniques similar to those contained in \cite[\S 3]{Blair2018}. Finally, in Appendix \ref{cluster_appendix} we prove the on-diagonal spectral cluster estimate used in Section \ref{spectralfunction}. The main components of the proof are extremely similar to arguments presented in Section \ref{asymptotics}, so we simply sketch the key points.

\subsection{Acknowledgments}
First and foremost, the author would like to thank his thesis advisor Y. Canzani for providing the inspiration for this project and for giving detailed feedback on several drafts of the article. The author is also grateful to J. Marzuola, J. Metcalfe, M. Taylor, and M. Williams for providing insight on various details throughout the course of this project. It is also a pleasure to thank G. Peccati and M. Rossi for some very interesting discussions regarding the applications of this work to monochromatic random waves. The author would also like to thank Y. Bonthonneau for some private communications which clarified a few details about the extension of B\'erard's original estimate to the case of manifolds without conjugate points. The author would like to thank both M. Blair and C. Sogge for their comments regarding the addition of derivatives to \thmref{logthm}. In particular, M. Blair had some insightful suggestions regarding the variation through geodesics argument in the proof of \lemref{u_nu_bound}. Finally, the author is tremendously grateful to the referee who reviewed the first version of this paper for providing detailed and helpful feedback, most notably a sketch of the proof of \lemref{exp_map}, which was a key component in adding derivatives to the main result. 




\section{The Spectral Function and the Wave Kernel}
\label{spectralfunction}
Since the spectral function $E_\lambda(x,y)$ is difficult to work with directly, we instead study its behavior by relating it to the kernel of $\cos( t\sqrt\Delta_g)$ via the Fourier transform, following techniques similar to those found in \cite{SoggeBook2014}. To accomplish this, let us note that 
\[
E_\lambda(x,y)  = \sum\limits_{j = 0}^\infty \mathds{1}_{[-\lambda,\lambda]}(\lambda_j)\varphi_j(x) {\varphi_j(y)},
\]
where $\mathds{1}_{[-\lambda,\lambda]}$ denotes the characteristic function of the interval $[-\lambda,\lambda]$. Since this characteristic function has Fourier transform $\int_{-\lambda}^\lambda e^{-it\tau}\mathds \,d\tau = \frac{2\sin (t\lambda)}{t}$, which is even, we can write
\begin{equation}\label{E_lambda_nocutoff}
E_\lambda(x,y) = \sum\limits_{j = 0}^\infty\frac{1}{\pi}\int\limits_{-\infty}^\infty\frac{\sin (t\lambda)}{t} \cos (t\lambda_j)\,\varphi_j(x) {\varphi_j(y)}\,dt,
\end{equation} 
where we can interpret the above integral as $\lim\limits_{N\to\infty}\frac{1}{\pi}\int_{-\infty}^\infty \beta(t/N)\frac{\sin( t\lambda)}{t}\cos (t\lambda_j)\,dt$ for any even function $\beta\in C_c^\infty(\R)$ with $\beta(0) = 1.$ This interpretation technically requires that $\lambda^2$ does not belong to the spectrum of $\Delta_g$, since
\[\lim\limits_{N\to\infty}\int\limits_{-\infty}^\infty\beta(t/N)\frac{\sin( t\lambda)}{t}\cos( t\lambda)\,dt = \frac{1}{2},\]
if $\beta$ is even, and so the limit does not actually recover $\mathds 1_{[-\lambda,\lambda]}(\lambda)$ (c.f. \cite{SoggeBook2014}). Thus, we will assume throughout the rest of this argument that $\lambda^2$ is not an eigenvalue. To justify this assumption, let us define the spectral cluster operator $E_{(\lambda,\lambda + A]}$ for $0 < A\le 1$ to be the orthogonal projection
\[
E_{_{\!(\lambda,\lambda+A]}}:L^2(M)\to \!\!\!\bigoplus\limits_{\lambda_j\in(\lambda,\lambda+A]} \!\!\!\ker (\Delta_g - \lambda_j^2)
\]
and so the corresponding Schwartz kernel is 
\begin{equation}\label{spectral_cluster_kernel}
E_{_{\!(\lambda,\lambda+A]}}(x,y) = \!\!\!\sum\limits_{\lambda_j\in(\lambda,\lambda+A]}\!\!\!\varphi_j(x) {\varphi_j(y)}.
\end{equation}
We then have the following estimate on derivatives of $E_{(\lambda,\lambda+A]}$ restricted to the diagonal, which is a generalization of results from \cite{Berard1977,SoggeBook2014}.
 
\begin{lem}
\label{spectral_cluster_log}
Let $(M,g)$ be as in Theorem \ref{logthm}. Then there are constants $\lambda_0,C_1,C_2>0$ such that 
\[
\sup\limits_{x\in M}\left|\partial_x^\alpha\partial_y^\alpha E_{_{\!(\lambda,\lambda+A]}}(x,y)\big|_{x=y}\right| \le C_1\lambda^{2|\alpha|}\left[ A\lambda^{n-1} + e^{C_2/A}A\max\{\lambda^{\frac{n-1}{2}},\lambda^{n-3}\}\right]
\]
for all $\lambda \ge \lambda_0$ and all $0 < A\le 1.$ In particular, if $A = \frac{1}{c\log\lambda}$ with $c > 0$ sufficiently small, then after possibly increasing $\lambda_0,$ we have
\[
\sup\limits_{x\in M}\left|\partial_x^\alpha\partial_y^\alpha E_{_{\!(\lambda,\lambda+A]}}(x,y)\big|_{y=x}\right| \le C \frac{\lambda^{n-1+2|\alpha|}}{\log\lambda}
\]
for all $\lambda \ge \lambda_0$ and for some $C > 0$.
\end{lem}
In the case where $|\alpha| = 0$ and $(M,g)$ has nonpositive curvature, this bound was formally stated in terms of spectral clusters in \cite{SoggeBook2014}, although the techniques required to prove it were first presented in \cite{Berard1977}. The result of \cite{Bonthonneau2017} can be easily used to extend the $|\alpha| = 0$ estimate to the case of manifolds with no conjugate points. The addition of derivatives is a new result, but we will postpone the proof, since it is largely a repetition of arguments found in Section \ref{asymptotics}.

It follows from Lemma \ref{spectral_cluster_log} that if $\lambda^2$ is in the spectrum of $\Delta_g$, we can shift to some slightly larger $\mu^2$ which is not an eigenvalue. The error introduced in doing so then satisfies
 
\begin{align*}
\left| \partial_x^\alpha\partial_y^\beta\lp E_{ \mu }(x,y)\! - \!E_{\lambda}(x,y)\rp\right| & \!\le\! \lp\sum\limits_{\lambda_j\in (\lambda,\mu]}\!\!|\partial_x^\alpha\varphi_j(x)|^2\rp^{1/2}\!\!\!\lp\sum\limits_{\lambda_j\in (\lambda,\mu]}\!\!|\partial_y^\beta\varphi_j(y)|^2\rp^{1/2}\\
& \le \frac{C\lambda^{n-1+|\alpha|+|\beta|}}{\log\lambda},
\end{align*}
 
\noindent provided that $|\mu  - \lambda| \le A$ for $A$ as above, which is always possible since the spectrum of $\Delta_g$ is discrete. 

 Now, formally interchanging the summation and the integral in \eqref{E_lambda_nocutoff} we would have 
\begin{equation}\label{formal_interchange}
E_\lambda(x,y) = \frac{1}{\pi}\int\limits_{-\infty}^\infty \frac{\sin (t\lambda)}{t}K(t,x,y)\,dt,
\end{equation}
where 
\[K(t,x,y) = \sum_{j=0}^\infty \cos (t\lambda_j) \,\varphi_j(x) {\varphi_j(y)}\]
 is the Schwartz kernel of $\cos (t\sqrt\Delta_g)$. This interchange is justified at the level of operator kernels if we allow $E_\lambda(x,y)$ to act on a $C^\infty$ function $f$ by integration in $y$. In this case the summation involves the Fourier coefficients of $f$, which have sufficient decay to guarantee that the sum converges absolutely, and thus we are justified in interchanging the sum and the integral. 

At this point it is convenient to introduce a smooth, even cutoff function $\wh\rho$ which will allow us to restrict the support of the integrand in \eqref{formal_interchange} to a region where we can approximate $K(t,x,y)$ by a parametrix. The error introduced in doing so can be controlled as follows.
 
\begin{prop}
\label{exact_kernel_remainder}
Let $(M,g)$ be as in Theorem \ref{logthm} and let $\wh\rho\in C_c^\infty(\R)$ be an even function with $\wh\rho(t) = 1$ for all $|t| < \frac{1}{2}\textnormal{inj}(M,g)$ and with support in $[-L,L]$ for some $L < \inj(M,g)$. Then, there exist constants $c,C,\lambda_0 > 0$ so that if $A = \frac{1}{c\log\lambda}$, we have 
\begin{equation}\label{exact_kernel_remainder_eqn}
\sup\limits_{x,y\in M}\left|\partial_x^\alpha\partial_y^\beta\lp E_\lambda(x,y) - \frac{1}{\pi}\int\limits_{-\infty}^\infty \wh\rho(At)\frac{\sin (t\lambda)}{t}K(t,x,y)\,dt\rp\right| \le \frac{C\lambda^{n-1+|\alpha|+|\beta|}}{\log\lambda}
\end{equation}
for all $\lambda \ge \lambda_0.$   
\end{prop}

\begin{proof}
We prove this result first for the case where $|\alpha| = |\beta| = 0.$ Observe that
\begin{align}\label{spectral_function_h_eqn}
\begin{split}
E_\lambda(x,y) - \frac{1}{\pi}\int\limits_{-\infty}^\infty \wh\rho(At)\frac{\sin (t\lambda)}{t}K(t,x,y)\,dt & = \sum\limits_{j=0}^\infty h_{\lambda,A}(\lambda_j)\varphi_j(x) {\varphi_j(y)},\\
\end{split}
\end{align}
where
\begin{equation}\label{remainder_function}
h_{\lambda,A}(\tau) = \mathds 1_{[-\lambda,\lambda]}(\tau) - \frac{1}{\pi}\int\limits_{-\infty}^\infty\wh \rho(At)\frac{\sin t\lambda}{t}\cos t\tau\,dt 
\end{equation}
for $\tau\in\R.$ We claim that $h_{\lambda,A}$ satisfies the bound 
\begin{equation}\label{remainder_function_bound}
|h_{\lambda,A}(\tau)| \le C_N\left(1 + A^{-1}\big| |\tau| - \lambda \big|\right)^{-N} 
\end{equation}
when $\lambda \ge 1$, for any $N = 1,2,3,\dotsc.$ To prove this, we note that if $\rho$ is the inverse Fourier transform of $\wh\rho$, then $\rho$ is an even Schwartz-class function with $\int \rho\,dt = \wh\rho(0) = 1$. Therefore,
\begin{align*}
\frac{1}{\pi}\int\limits_{-\infty}^\infty \wh\rho(At)\frac{\sin t\lambda}{t}\cos t\tau\,dt = \int\limits_{-\infty}^\infty \frac{1}{A}\rho\lp\frac{\tau-s}{A}\rp \mathds 1_{[-\lambda,\lambda]}(s)\,ds = \int\limits_{\frac{\tau - \lambda}{A}}^{\frac{\tau + \lambda}{A}}\rho(s)\,ds.
\end{align*}
When $|\tau| \gg \lambda$, we use the fact that $\rho$ is rapidly decaying and $\mathds 1_{[-\lambda,\lambda]}(\tau)$ is zero. When $\lambda \gg |\tau|$, we use that $\rho$ decays rapidly and integrates to one and that $\mathds 1_{[-\lambda,\lambda]}$ is identically one on its support. These facts combine to give \eqref{remainder_function_bound}.

We can therefore control the right-hand side of \eqref{spectral_function_h_eqn} using bounds on $h_{\lambda,A}$, along with Lemma \ref{spectral_cluster_log}. For this we break the summation into intervals of size $A$ as follows. For each $N > 0$, there exists a $C_N > 0$ so that 
\begin{equation}
\label{spectral_windows}
\left|\sum\limits_{j=0}^\infty h_{\lambda,A}(\lambda_j) \varphi_j(x) {\varphi_j(y)}\right|
\le \sum\limits_{k=0}^\infty\sum\limits_{\lambda_j\in[kA,(k+1)A]} \hspace{-.25in}C_N(1 + A^{-1}\big| \lambda - \lambda_j\big|)^{-N}\left|\varphi_j(x)\varphi_j(y)\right|
\end{equation}
by \eqref{remainder_function_bound}. In each interval, we can write $\lambda_j = As_j$ for some $s_j\in[k,k+1]$, and hence 
\[(1 + A^{-1}|\lambda - \lambda_j|)^{-N} = (1 + |A^{-1}\lambda - s_j|)^{-N} \le C_N(1 + |A^{-1}\lambda - k|)^{-N},\]
for some possibly larger $C_N>0,$ so we can use the triangle inequality to bound the right-hand side of \eqref{spectral_windows} by
\begin{equation}\label{k_summation}
\sum\limits_{k=0}^\infty \lp C_N(1 + |A^{-1}\lambda -k|)^{-N}\hspace{-2em}\sum\limits_{\lambda_j\in[kA\,,\,(k + 1)A]}\hspace{-1em}\left| \varphi_j(x) {\varphi_j(y)}\right|\rp.
\end{equation} 
Next, we seek to apply Lemma \ref{spectral_cluster_log} to each of the sums over $\lambda_j\in[kA,(k+1)A]$ with $\lambda = kA$. However, we must first discard all terms for which $kA \le \lambda_0$, where $\lambda_0$ is as in the statement of Lemma \ref{spectral_cluster_log}. To see that this is possible, observe that
\begin{equation}\label{small_k}
\sum\limits_{k\in\left[0,\frac{\lambda_0}{A}\right]}\sum\limits_{\lambda_j\in[kA\,,\,(k+1)A]}|\varphi_j(x) {\varphi_j(y)}| \le \sum\limits_{k\in\left[0,\frac{\lambda_0}{A}\right]}\sum\limits_{\lambda_j\in[0,\lambda_0+1]}|\varphi_j(x) {\varphi_j(y)}| \le \frac{C}{A},
\end{equation}
for some constant $C> 0$, since $(k+1)A \le \lambda_0 + 1$, the set $\{j:\,\lambda_j \le\lambda_0+1\}$ is finite, and each $\varphi_j$ is bounded. Note that here $C$ may depend on $\lambda_0$, but not on $A$.

Then, for all $k$ with $k \ge \frac{\lambda_0}{A}$, we have by Lemma \ref{spectral_cluster_log} and Cauchy-Schwarz that
\begin{equation}\label{k_cluster}
\sum\limits_{\lambda_j\in[kA,(k+1)A]}\hspace{-2em}\left|\varphi_j(x)\varphi_j(y)\right|
 \le C_1 \left[ A^nk^{n-1} + e^{C_2/A}\max\{A^{\frac{n+1}{2}}k^{\frac{n-1}{2}},A^{n-2}k^{n-3}\}\right].
\end{equation}
By Corollary \ref{localized_sums} we have for sufficiently large $N$ that
\[
\sum\limits_{k\ge \frac{\lambda_0}{A}}^\infty  C_N(1 + |A^{-1}\lambda -k|)^{-N}A^nk^{n-1} \le  \wt C_NA^n(A^{-1}\lambda)^{n-1} = \wt C_N A\lambda^{n-1},
\]
for some $\wt C_N > 0$. This is because the factor of $(1 + \big|A^{-1}\lambda -k\big|)^{-N}$ serves to localize the sum to the region where $k\approx A^{-1}\lambda$. Analogously, after potentially increasing $\wt C_N$, we have
\[
\sum\limits_{k\ge \frac{\lambda_0}{A}}^\infty  C_N(1 + |A^{-1}\lambda -k|)^{-N}e^{C_2/A}A^{\frac{n+1}{2}}k^{\frac{n-1}{2}} \le \wt C_N Ae^{C_2/A}\lambda^{\frac{n-1}{2}}
\]
and
\[
\sum\limits_{k\ge \frac{\lambda_0}{A}}^\infty  C_N(1 + |A^{-1}\lambda -k|)^{-N}e^{C_2/A}A^{n-2}k^{n-3} \le \wt C_N Ae^{C_2/A}\lambda^{n-3}.
\]

Therefore, by the above estimates and \eqref{k_cluster}, there is some $\wt C_N> 0$ so that
\[
\sum\limits_{k\ge \frac{\lambda_0}{A}}^\infty  \lp C_N(1 + |A^{-1}\lambda -k|)^{-N}\hspace{-2.2em}\sum\limits_{\lambda_j\in[kA\,,\,(k+1)A]}\hspace{-1em}\left| \varphi_j(x) {\varphi_j(y)}\right|\rp \le \wt C_N \left[ A\lambda^{n-1} + Ae^{C_2/A}\max\{\lambda^{\frac{n-1}{2}},\lambda^{n-3}\}\right].
\]

\noindent Now, if we take $A = \frac{1}{c\log\lambda}$ for $c > 0$, we have that $e^{C_2/A} = \lambda^{c C_2}$. Hence, if $c$ is chosen small enough, and if we increase $\lambda_0$ so that $A = \frac{1}{c\log\lambda} \le 1$ when $\lambda \ge \lambda_0$, we have 
\begin{equation}\label{large_k}
\sum\limits_{k\ge \frac{\lambda_0}{A}}^\infty\lp  C_N(1 + |A^{-1}\lambda -k|)^{-N}\hspace{-2em}\sum\limits_{\lambda_j\in[kA\,,\,(k+1)A]}\hspace{-1em} \left|\varphi_j(x) {\varphi_j(y)}\right| \rp\le \wt C_N\frac{\lambda^{n-1}}{\log\lambda},
\end{equation}
for all $\lambda \ge \lambda_0$ after possibly once again modifying $\wt C_N$. Picking some fixed $N$ large enough and combining \eqref{large_k} with \eqref{small_k}, we obtain
\[
\left|\sum\limits_{j=0}^\infty h_{\lambda,A}(\lambda_j) \varphi_j(x) {\varphi_j(y)}\right| \le \wt C_N \frac{\lambda^{n-1}}{\log\lambda} + C\log\lambda
\]
when $\lambda \ge \lambda_0$, since $A = \frac{1}{c\log\lambda}$. Note that since $n\ge 2$, the $\mathcal O\lp\frac{\lambda^{n-1}}{\log\lambda}\rp$ term dominates the $\mathcal O(\log\lambda)$ term as $\lambda \to\infty$, and hence we can choose some $\wt \lambda_0 \ge \lambda_0$ such that
\[\left|\sum\limits_{j=0}^\infty h_{\lambda,A}(\lambda_j) \varphi_j(x) {\varphi_j(y)}\right| \le \frac{C\lambda^{n-1}}{\log\lambda}\]
for all $\lambda \ge \wt \lambda_0$ and some $C > 0$.

To include $\partial_x^\alpha\partial_y^\beta$, we simply apply the estimate from \lemref{spectral_cluster_log} to obtain the appropriate modification of \eqref{k_cluster}, which is given by 
\[
\sum\limits_{\lambda_j\in[kA,(k+1)A]}\hspace{-1em}\left|\partial_x^\alpha\varphi_j(x)\partial_y^\beta\varphi_j(y)\right|\le C_1 \lambda^{|\alpha|+|\beta|}\left[ A^nk^{n-1} + e^{C_2/A}\max\{A^{\frac{n+1}{2}}k^{\frac{n-1}{2}},A^{n-2}k^{n-3}\}\right],
\]
 which only serves to increase the relevant powers of $\lambda$ by $|\alpha| +|\beta|$, and hence the proof goes through with no further adjustments.
 
\end{proof}

With Proposition \ref{exact_kernel_remainder} in hand, it now suffices to show that the integral in \eqref{exact_kernel_remainder_eqn} has the asymptotic behavior that we claimed in Theorem \ref{logthm}. To accomplish this, we use the Hadamard parametrix to approximate the cosine kernel, which we discuss in the following section.




\section{Approximation via the Hadamard parametrix}
\label{hadamard}

Given Proposition \ref{exact_kernel_remainder}, the proof of Theorem \ref{logthm} would be complete if we could show that for every $\alpha,\beta$, there exists $C,c > 0$ such that for all $\lambda$ sufficiently large, the remainder

\begin{equation}\label{remainder_K}
R_K(x,y,\lambda)\!\!:= \!\frac{1}{\pi}\!\!\!\int\limits_{-\infty}^\infty \!\!\wh\rho(At)\frac{\sin t\lambda}{t} K(t,x,y)\,dt - \frac{1}{(2\pi)^n}\!\!\!\!\!\!\!\!\int\limits_{|\xi|_{g_x^{-1}}\le \lambda}\!\!\!\!\!\!\! e^{i\langle \exp_x^{-1}(y),\xi\rangle_{g_x^{-1}}}\frac{d\xi}{\sqrt{\det g_x}}
\end{equation}
 satisfies
\begin{equation}\label{true_kernel_estimate}
\sup\limits_{d_g(x,y)\le \frac{1}{2}\inj(M,g)}\left|\partial_x^\alpha\partial_y^\beta R_K(x,y,\lambda)\right|\le\!\! \frac{C\lambda^{n-1+|\alpha| + |\beta|}}{\log\lambda}
\end{equation}
 when $A = \frac{1}{c\log\lambda}$. However, since it is not possible to compute $K(t,x,y)$ exactly, we instead approximate it using the Hadamard parametrix. In fact, as in \cite{Berard1977}, we will use the assumption of no conjugate points to lift to the universal cover of $M$ to ensure that the parametrix exists for large $|t|$. Our ability to control the parametrix for timescales on the order of $\log\lambda$ is what will allow us to estimate the integral involving $K(t,x,y)$ in \eqref{remainder_K} for $A = \frac{1}{c\log\lambda}$, since the integrand is supported where $t\in[-1/A,1/A]\approx [-\log\lambda,\log\lambda]$. The first part of this section consists of a summary of results about the Hadamard parametrix which are proved in other works, and we refer the reader to the appropriate sources which contain the corresponding details. Afterward, we prove that the error introduced in replacing $K(t,x,y)$ by a partial sum of the parametrix in \eqref{remainder_K} is sufficiently small, and we discuss some particular formulas for the parametrix terms which will be very useful when we wish to do the explicit asymptotic analysis in Section \ref{asymptotics}.

Since $(M,g)$ has no conjugate points, we know that for a fixed $x_0\in M$ the exponential map 
\[p:= \exp_{x_0}: T_{x_0}M\to M\]
 is a covering map, and hence $\wt M := T_{x_0}M\cong \R^n$ is the universal cover of $M$ when equipped with the metric $\wt g = p^*g.$ If we denote by $\Gamma$ the deck transformation group of isometries on $\wt M$ corresponding to $p$, the work of \cite{Berard1977} shows that the wave kernel $K(t,x,y)$ on the base manifold $M$ has an expansion of the form
\begin{equation}\label{parametrix}
K(t,x,y) = \sum\limits_{\nu = 0}^\infty\sum\limits_{\gamma\in\Gamma} u_\nu(\wt x,\gamma\wt y)\partial_t W_\nu(t,d_{\wt g}(\wt x,\gamma\wt y)) \, \mod C^\infty,
\end{equation}
where $\wt x,\wt y$ are any chosen lifts of $x,y\in M$. The coefficient functions $u_\nu$ are defined for any $\wt x,\wt y\in \wt M$ by 
\begin{equation}\label{parametrix_coeffs}
\begin{cases}
u_0(\wt x,\wt y)= \Theta^{-\frac{1}{2}}(\wt x,\wt y)\\
u_\nu(\wt x,\wt y) = \Theta^{-\frac{1}{2}}(\wt x,\wt y)\int\limits_0^1 s^{\nu-1}\Theta^{1/2}(\wt x,\alpha_{\wt x\wt y}(s))\Delta_{\wt g,\wt y} u_{\nu-1}(\wt x,\alpha_{\wt x\wt y}(s))\,ds, \hskip 0.1in \nu \ge 1,
\end{cases}
\end{equation}
where $\Theta(\wt x,\wt y) = |\det D_{\exp_{\wt x}^{-1}(\wt y)}\exp_{\wt x}|$ and $\alpha_{\wt x\wt y}$ is the unique minimizing geodesic in $(\wt M,\wt g)$ connecting $\wt x$ and $\wt y$ parametrized by arc length, which exists because the metric on $\wt M$ is uniquely geodesic. In $\R^n$, the distributions $W_\nu$ for $\nu = 0,1,2,\dotsc,$ are defined by
\begin{equation}\label{parametrix_terms_definition}
W_\nu(t,|w|) = \frac{\nu!}{(2\pi)^{n+1}}\lim\limits_{\ve\to 0^+}\int\limits_{\R^{n+1}} e^{i\langle w,\xi\rangle + it\tau}(|\xi|^2 -(\tau - i\ve)^2)^{-\nu-1}\,d\xi\,d\tau,
\end{equation}
for $w\in\R^n$ and $t > 0$. At $t = 0$, we have $W_\nu(0+,|w|) = \lim_{t\to0^+}W_\nu(t,|w|) = 0$ for all $\nu\ge0$ by \cite[Prop 1.2.4]{SoggeBook2014}. We then extend each distribution to $t\in\R$ by imposing the condition $W_\nu(-t,|w|) = -W_\nu(t,|w|)$ so that $W_\nu$ is odd in $t$.  It is clear from the definition that $W_\nu$ depends only on the norm of $w$, since it is the inverse Fourier transform of a radial distribution in $\xi$. It is also easy to verify from \eqref{parametrix_terms_definition} that $W_\nu$ is homogeneous of degree $2\nu -n +1$. Furthermore, as $\nu$ increases,  the extra decay of the integrand in $(\tau,\xi)$ results in additional regularity in $(t,w)$. In particular, we have that if $\nu > k + \frac{n-1}{2}$ for some integer $k$, then $W_\nu$ is a continuous function whose derivatives up to order $k$ are continuous \cite[\S 17.4]{HormanderBook1985}. One can then pull back via geodesic normal coordinates centered at $\wt x\in \wt M$ to obtain distributions $W_\nu(t,d_{\wt g}(\wt x,\wt y))$ defined on $\R\times \wt M\times\wt M$ (see \cite[\S 17.4]{HormanderBook1985} and \cite[\S 2.4]{SoggeBook2014} for details). Note that we use $\partial_t W_\nu$ in \eqref{parametrix}, rather than $W_\nu$ itself. This is due to the fact that the parametrix construction is generally done first for the kernel of $\frac{\sin(t\sqrt{\Delta_{\smash g}})}{\sqrt{\Delta_{\smash g}}},$ and then the parametrix for $\cos(t\sqrt{\Delta_{\smash{g}}})$ is obtained by differentiating in $t$. 

The sum over $\gamma\in \Gamma$ in \eqref{parametrix} is finite for any fixed $t$, since the wave equation has finite speed of propagation. Indeed, is a consequence of the Paley-Weiner theorem that $W_\nu(t,d_{\wt g}(\wt x,\wt y))$ is supported in the light cone $\{(t,\wt x,\wt y)\in\R\times\wt M\times \wt M:\,d_{\wt g}(\wt x,\wt y) \le |t|\}$. Additionally, by \cite[Lemma 6]{ColindeVerdiere1973}, we have that for any $\wt x,\wt y\in \wt M$,
\begin{equation}\label{gamma_count}
\#\{\gamma\in \Gamma: d_{\wt g}(\wt x,\gamma \wt y) \le |t|\} \le C_1 e^{C_2 |t|},
\end{equation}
where $C_1,C_2$ are positive constants which are independent of $\wt x,\wt y$. Therefore, at most $C_1e ^{C_2 |t|}$ terms in the sum over $\gamma\in \Gamma$ in \eqref{parametrix} are nonzero for any fixed $t$. We note that this result was stated in \cite{ColindeVerdiere1973} for $(M,g)$ having negative sectional curvature, but the proof only depends on the fact that the Ricci curvature of $(\wt M,\wt g)$ is bounded below.

Since we wish to use the parametrix instead of the exact wave kernel in the integral in \eqref{remainder_K}, we must estimate the difference between them. For any fixed $N \ge 0$ and $x,y\in M$, define
\begin{equation}\label{parametrix_partial_sums}
K_N(t,x,y) = \sum\limits_{\nu = 0}^N \sum\limits_{\gamma\in \Gamma}u_\nu(\wt x,\gamma\wt y)\partial_tW_\nu(t,d_{\wt g}(\wt x,\gamma \wt y)).
\end{equation}
The following proposition estimates the error introduced by using $K_N$ in place of $K$ in \eqref{remainder_K}, which is generalizes a result from \cite{Berard1977} to include derivatives in $x$ and $y$.

\begin{prop}\label{parametrix_error}
Let $(M,g)$ be as in Theorem 1, and let $\wh\rho\in C_c^\infty(\R)$ be as in \propref{exact_kernel_remainder}. Let $K$ be the kernel of $\cos(t\sqrt{\Delta_g})$ and let $K_N$ be defined by \eqref{parametrix_partial_sums}. If $\alpha,\beta$ are multi-indices and if $N> m + |\alpha| +\frac{n+1}{2}$ for some integer $m > \frac{n}{2} +|\beta| - 1$, then there exist constants $C_1,C_2>0$ so that for any $0 < A \le 1,$ we have
\begin{equation}\label{parametrix_error_eqn}
\sup\limits_{x,y\in M}\left|\frac{1}{\pi}\int\limits_{-\infty}^\infty \wh\rho(At)\frac{\sin t\lambda}{t}\partial_x^\alpha\partial_y^\beta\lp K_N(t,x,y) - K(t,x,y)\rp\,dt\right| \le C_1 e^{C_2/A}
\end{equation}
for all $\lambda >0$.
\end{prop}
\begin{proof}
Since $\wh\rho(At)$ is uniformly bounded and equal to zero outside the interval $t\in[-1/A,1/A]$, the above estimate would follow immediately from the bound
 
\begin{equation}\label{pointwise_remainder_bound}
\sup\limits_{x,y\in M}\left|\frac{1}{t}\partial_x^\alpha\partial_y^\beta\lp K_N(t,x,y) - K(t,x,y)\rp\right| \le C_1 e^{C_2 |t|}.
\end{equation}
 
We prove this bound using some standard energy inequalities for the wave equation and a Sobolev embedding, along with some pointwise bounds on derivatives of $u_\nu$ and $\partial_tW_\nu$ which are direct consequences of results from Appendix \ref{geometric_estimates}. The Hadamard parametrix construction in \cite{Berard1977} shows that the remainder
\[R_N(t,x,y) = K_N(t,x,y) - K(t,x,y)\]
satisfies an inhomogeneous wave equation of the form 
\[\begin{cases}
(\partial_t^2 + \Delta_{g,y})R_N(t,x,y) = F_N(t,x,y),\\
R_N(0,x,y) = 0\\
\partial_t R_N(0,x,y) = 0,
\end{cases}\]
where $F_N(t,x,y) = C\sum\limits_{\gamma\in\Gamma}(\Delta_{\wt g,\wt y}u_N(\wt x,\gamma\wt y))\partial_t W_N(t,d_{\wt g}(\wt x,\gamma\wt y))$ for any lifts $\wt x,\wt y$ of $x,y$ and some constant $C$, and $F_N$ is of class $C^{m+|\alpha|}$, provided $N > m + |\alpha| +\frac{n+1}{2}$, . Noting that derivatives in $x$ commute with $\Delta_{g,y}$, we have that 
\[\begin{cases}
(\partial_t^2 + \Delta_{g,y})(\partial_x^\alpha R_N(t,x,y)) = \partial_x^\alpha F_N(t,x,y)\\
\partial_x^\alpha R_N(0,x,y) = 0\\
\partial_t \lp\partial_x^\alpha R_N(0,x,y) \rp= 0.
\end{cases}\] 

A standard energy inequality for wave equations with vanishing initial data (see \cite[Ch. 47]{TrevesBook1975}) yields that for any $x\in M$ and $t > 0$, 
\begin{equation}\label{sobolev_remainder}
 \left\|\partial_x^\alpha R_N(t,x,\cdot)\right\|_{H^{m+1}(M)} \le C_1 e^{C_2 t}\int\limits_0^t \|\partial_x^\alpha F_N(s,x,\cdot)\|_{H^m(M)}\,ds,
\end{equation}
for some constants $C_1,C_2>0,$ where $H^m(M)$ is the standard $L^2$-based Sobolev space of order $m$. By hypothesis, $m + 1 > \frac{n}{2} + |\beta|$, and hence by Sobolev embedding, we have
\begin{equation}\label{sobolev_embedding}
\sup\limits_{y\in M}|\partial_x^\alpha\partial_y^\beta R_N(t,x,y)| \le C_1 e^{C_2 t}\int\limits_0^t \|\partial_x^\alpha F_N(s,x,\cdot)\|_{H^m(M)}\,ds,
\end{equation}
for some possibly different $C_1,C_2 > 0.$ 

In order to analyze $\partial_x^\alpha F_N(t,x,y)$, we must first identify $\partial_x^\alpha$ with an operation on the cover, which we can accomplish by locally pulling back via the covering map $p.$ To be more precise, if we fix $\wt x\in \wt M$, we can identify a small enough coordinate patch $U_{\wt x}$ containing $\wt x$ with a coordinate patch on $M$, since $p|_{U_{\wt x}}$ is an isometry, and therefore invertible, if $U_{\wt x}$ is small enough. Thus, if $\partial_x^\alpha$ indicates differentiation in the coordinates on $M$, we can identify it with an operator $P_{\wt x}$ involving only differentiation in the coordinates on $\wt M$ and derivatives of $p|_{U_{\wt x}}^{-1}$. Since $p$ is a local isometry and $M$ is compact, we have that $P_{\wt x}\in \textnormal{Diff}(\wt M)$, where $\textnormal{Diff}(\wt M)$ denotes the algebra of $C^\infty$-bounded differential operators on $\wt M,$ defined as in \cite[Appendix A.1]{Shubin1992}. That is, we say that $P_{\wt x}$ is a $C^\infty$-bounded differential operator of order $k$ if for some fixed $r\in (0,\inj(\wt M))$, we can express $P_{\wt x}$ as 
\[\sum\limits_{|\sigma|\le k} a_\sigma(\wt x)\partial_{\wt x}^\sigma\]
in any canonical coordinate neighborhood of radius $r$, where the $a_\sigma$ are smooth functions with $|\partial_{\wt x}^\alpha a_\sigma(\wt x)| \le C_{\alpha}$ for all $\alpha$, and the constant is independent of the choice of coordinate neighborhood. Thus, we may interpret $\partial_x^\alpha F_N(t,x,y)$ as 
\[C\sum\limits_{\gamma\in\Gamma}P_{\wt x}\left[(\Delta_{\wt g,\wt y}u_N(\wt x,\gamma\wt y))\partial_t W_N(t,d_{\wt g}(\wt x,\gamma\wt y))\right].\]
Recalling \eqref{gamma_count}, the definition of $H^m$, and the fact that $\partial_t W_N$ is supported where $d_{\wt g}(\wt x,\wt y)\le |t|$, we have that for $t > 0$,
\begin{align}\label{F_N_bound}
\begin{split}
\|\partial_x^\alpha F_N(t,x,\cdot)\|_{H^m(M)} &\le C\sum\limits_{\gamma\in\Gamma}\|(1 + \Delta_{\wt g,\wt y})^{m/2}P_{\wt x}\left[(\Delta_{\wt g,\wt y}u_\nu(\wt x,\gamma\cdot))\partial_t W_N(t,d_{\wt g}(\wt x,\gamma\cdot))\right]\|_{L^2(\wt M)}\\
 &\le  C_1 e^{C_2t}\big\|(1 + \Delta_{\wt g,\wt y})^{m/2}P_{\wt x}\left[(\Delta_{\wt g,\wt y}u_N(\wt x,\cdot))\partial_t W_N(t,d_{\wt g}(\wt x,\cdot))\right]\big\|_{L^2(\wt M)},
\end{split}
\end{align}
since $\Delta_{\wt g,\wt y}$ commutes with isometries acting in the $\wt y$ variable. We claim that the function inside the $L^2$ norm on the right-hand side is bounded pointwise by a constant multiple of $e^{C_3 s}\mathds 1_{[0,s]}(d_{\wt g}(\wt x,\cdot))$ for some $C_3 > 0.$ Since $\Delta_{\wt g,\wt y}\in \textnormal{Diff}(\wt M)$, it will suffice to show that for any $P_{\wt x},Q_{\wt y}\in \textnormal{Diff}(\wt M)$, 
\begin{equation}\label{coeff_bound}
|P_{\wt x}Q_{\wt y}u_N(\wt x,\wt y)| \le C' e^{C'' d_{\wt g}(\wt x,\wt y)},
\end{equation}
and 
\begin{equation}\label{W_nu_bound}
|P_{\wt x}Q_{\wt y}\partial_t W_N(s,d_{\wt g}(\wt x,\wt y))|\le C'e^{C''s}\mathds 1_{[0,s]}(d_{\wt g}(\wt x,\wt y)),
\end{equation}
for some $C',C'' > 0$ which may depend on $N$, $P_{\wt x}$, and $Q_{\wt y}$. Inequality \eqref{coeff_bound} is exactly the content of \lemref{u_nu_bound}, which is proved in Appendix \ref{geometric_estimates}, so we need only show \eqref{W_nu_bound}. For this, we use the observation from \cite[\S 17.4]{HormanderBook1985} that $W_N(s,d_{\wt g}(\wt x,\wt y))$ is a constant multiple of $(s^2 - d_{\wt g}(\wt x,\wt y)^2)_+^{N - \frac{n-1}{2}}$ . Our hypotheses ensure that $N$ is sufficiently large so that $W_N$ remains a continuous function after applying $\partial_t,$ $P_{\wt x}$, and $Q_{\wt y}$. Since factors of $d_{\wt g}(\wt x,\wt y)^2$ may appear due to the chain rule, we must apply \lemref{exp_map} to control the derivatives of these factors. We then have that $P_{\wt x}Q_{\wt y}\partial_t W_N(s,d_{\wt g}(\wt x,\wt y))$ exhibits at most exponential growth in $d_{\wt g}(\wt x,\wt y)$ and depends polynomially on $s$. Recalling that $W_N$ is supported where $d_{\wt g}(\wt x,\wt y) \le s$ gives \eqref{W_nu_bound}. 

Combining \eqref{coeff_bound} and \eqref{W_nu_bound} with \eqref{sobolev_embedding} and \eqref{F_N_bound}, we obtain
\[\sup\limits_{y\in M}|\partial_x^\alpha\partial_y^\beta R_N(t,x,y)| \le C_1 e^{C_2 t}\int\limits_0^t e^{C_3s}\|\mathds 1_{[0,s]}(d_{\wt g}(\wt x,\cdot))\|_{L^2(\wt M)}\,ds.\]
Since the curvature of $\wt M$ is bounded below, the volume of the geodesic ball centered at $\wt x$ of radius $s$ can grow at most exponentially fast in $s$ with constants independent of $\wt x$, and hence we have
\[\sup\limits_{x,y\in M}|\partial_x^\alpha\partial_y^\beta R_N(t,x,y)| \le C_1 e^{C_2 t}\]
after possibly increasing $C_1$ and $C_2.$ Recalling that $R_N$ and $\partial_t R_N$ vanish as $t\to 0^+$ and that $R_N$ is even with respect to $t$, we can also write 
\[\sup\limits_{x,y\in M}\left|\frac{1}{t}\partial_x^\alpha\partial_y^\beta R_N(t,x,y)\right| \le C_1 e^{C_2| t|},\]
for $t\in\R$, which is exactly \eqref{pointwise_remainder_bound}, and so the proof is complete.
\end{proof}
Before we explicitly estimate the integral involving $K(t,x,y)$ in \eqref{remainder_K}, we take note of another formula for $\partial_t W_\nu$. By \eqref{parametrix_terms_definition} and standard Fourier transform techniques, we have that $W_0(t,|w|)$ for $w\in \R^n$ solves ${(\partial_t^2 + \Delta_{\R^n})W_0(t,|w|) = 0}$ with initial conditions ${W_0(0,|w|) = 0},$ $
\partial_t W_0(0,|w|) = \delta(w),$ where $\delta$ is the Dirac distribution centered at $w = 0.$ Since $W_0(t,|w|)$ is supported in the union of the forward and backward light cones, we have by uniqueness of solutions to the wave equation that 
\[
W_0(t,|w|) = \frac{1}{(2\pi)^n}\int\limits_{\R^n}e^{i\langle w,\xi\rangle}\frac{\sin (t|\xi|)}{|\xi|}\,d\xi
\]
and thus 
\begin{equation}
\partial_tW_0(t,|w|) = \frac{1}{(2\pi)^n} \int\limits_{\R^n}e^{i\langle w,\xi\rangle}\cos (t|\xi|)\,d\xi.
\end{equation}
It is a straightforward calculation to see from \eqref{parametrix_terms_definition} that $\partial_t W_\nu = \frac{t}{2}W_{\nu-1}$ for any $\nu \ge 1$, and hence one can use integration by parts and induction to show that 
\begin{equation}
\partial_tW_\nu(t,|w|) = \sum\limits_{j+k= \nu-1}\sum\limits_{\pm} \frac{C^\pm_{j,k}}{(2\pi)^n}\int\limits_{\R^n} e^{i\langle w,\xi\rangle\pm it|\xi|} t^{j+1} |\xi|^{-\nu-k}\,d\xi,
\end{equation}
where $j,k$ are nonnegative integers, the $C^\pm_{j,k}$ are some constants depending only on $j,$ $k$, and $\nu$ \cite[Rmk 1.2.5]{SoggeBook2014}. Here we interpret each term in the sense of Fourier integral operators. We note that the above formula is singular at $\xi = 0,$ but this is of little consequence for our application. To see this, we may introduce a smooth cutoff function $\chi\in C_c^\infty(\R)$ such that $\chi \equiv 0$ on $[-1,1]$ and $\chi \equiv 1$ outside $[-2,2]$. Then 
\begin{equation}\label{near_zero_cutoff}
\int\limits_{\R^n}e^{i\langle w,\xi\rangle\pm i t|\xi|}(1 - \chi(|\xi|))|\xi|^{-\nu-k}\,d\xi
\end{equation}
is the inverse Fourier transform of a family of compactly supported distributions in $\xi$ which depends in a smooth and bounded way on $t\in\R.$ Recall that the Fourier transform maps $\mathscr E'(\R^n)\to C^\infty(\R^n)$ and $\mathscr S'(\R^n)\to\mathscr S'(\R^n),$ where $\mathscr E'(\R^n)$ denotes the space of compactly supported distributions and $\mathscr S'(\R^n)$ denotes the space of tempered distributions. Since $e^{\pm it|\xi|}(1 - \chi(|\xi|))|\xi|^{-\nu-k}$ lies in the intersection of $\mathscr E'$ and $\mathscr S'$, we see that \eqref{near_zero_cutoff} is therefore a smooth and tempered function of $(t,w).$ Thus, we can write 
\begin{equation}
\partial_tW_\nu(t,|w|) = \!\!\!\!\!\sum\limits_{j+k= \nu-1}\sum\limits_{\pm} \frac{C^\pm_{j,k}}{(2\pi)^n}\int\limits_{\R^n} e^{i\langle w,\xi\rangle\pm it|\xi|} t^{j+1} |\xi|^{-\nu-k}\chi(|\xi|)\,d\xi + f_\nu(t,w),
\end{equation}
for some $f_\nu:\R\times \R^n \to \C$ which is smooth and tempered as a function of $(t,w)$. Pulling back via the inverse exponential map $\exp_{\wt x}^{-1}:\wt M\to T_{\wt x}^*\wt M$ then gives
\begin{align}\begin{split}\label{linear_comb}
& \partial_tW_\nu(t,d_{\wt g}(\wt x,\wt y))=\!\!\!\!\sum\limits_{j+k= \nu-1}\!\sum\limits_{\pm} \frac{C^\pm_{j,k}}{(2\pi)^n}\!\!\int\limits_{T_{\wt x}^*M}\!\!\!\! e^{i\langle \exp_{\wt x}^{-1}(\wt y),\xi\rangle_{\wt g}\pm it|\xi|} t^{j+1} |\xi|_{\wt g}^{-\nu-k}\frac{\chi(|\xi|)d\xi}{\sqrt{\det\wt g_{\wt x}}}\\
& \hspace{1.3in} + f_\nu(t,\exp_{\wt x}^{-1}(\wt y)).
\end{split}
\end{align}
 Here we recall that $\langle\cdot,\cdot\rangle_{\wt g}$ and $|\cdot|_{\wt g}$ are taken to mean the inner product and norm on the cotangent fibers, respectively.  Similarly pulling back the formula for $\partial_t W_0$, we obtain
\begin{equation}\label{zero_term_cosine}
\partial_tW_0(t,d_{\wt g}(\wt x,\wt y)) = \frac{1}{(2\pi)^n} \int\limits_{T_{\wt x}^*M}e^{i\langle \exp_{\wt x}^{-1}(\wt y),\xi\rangle_{\wt g}}\cos (t|\xi|_{\wt g})\frac{d\xi}{\sqrt{|\wt g_x|}}.
\end{equation}
We make extensive use of formulas \eqref{linear_comb} and \eqref{zero_term_cosine} in Section \ref{asymptotics}.




\section{Explicit Asymptotics}
\label{asymptotics}
By taking $A = \frac{1}{c\log\lambda}$ in Proposition \ref{parametrix_error} for $c$ small enough and combining it with Proposition \ref{exact_kernel_remainder}, we have reduced the proof of Theorem \ref{logthm} to showing that the following estimate holds. This is because the $C_1e^{C_2/A} = C_1\lambda^{cC_2}$ error bound in Proposition \ref{parametrix_error} is much smaller than $\mathcal O\lp\frac{\lambda^{n-1+|\alpha|+|\beta|}}{\log\lambda}\rp$ for $c$ small and $\lambda$ large. 

\begin{prop}\label{parametrix_estimate}
Let $(M,g)$ be as in Theorem \ref{logthm} and fix $\wh\rho\in C_c^\infty(\R)$ as in \propref{exact_kernel_remainder}. Then, for any integer $N\ge 0$ and any multi-indices $\alpha,\beta$, there exist positive constants $c,C,\lambda_0$ so that if $A = \frac{1}{c\log\lambda}$, then
\begin{align}\label{parametrix_estimate_eqn}
\begin{split}
\frac{1}{\pi}\int\limits_{-\infty}^\infty \wh\rho(At) \frac{\sin (t\lambda)}{t}K_N(t,x,y)\,dt & = \frac{\lambda^n}{(2\pi)^n}\int\limits_{B_x^*M} \!\!\!\! e^{i\lambda\langle \exp_x^{-1}(y),\xi\rangle_{\wt g}}\frac{d\xi}{\sqrt{\det g_x}} + R_{N,A}(x,y,\lambda),
\end{split}
\end{align}
where 
\[\sup\limits_{d_g(x,y)\le \frac{1}{2}\inj(M,g)}|\partial_x^\alpha\partial_y^\beta R_{N,A}(x,y,\lambda)|\le \frac{C\lambda^{n-1+|\alpha|+|\beta|}}{\log\lambda}\]
for all $\lambda \ge\lambda_0$.
\end{prop}

Recalling the definition of $K_N$ from \eqref{parametrix_partial_sums}, we have that the left-hand side of \eqref{parametrix_estimate_eqn} can be written as  
\begin{equation}\label{double_sum}
\sum\limits_{\gamma\in\Gamma}\sum\limits_{\nu = 0}^N \pi^{-1}u_\nu(\wt x,\gamma\wt y)\int\limits_{-\infty}^\infty\wh\rho(At) \frac{\sin t\lambda}{t}\partial_t W_\nu(t,d_{\wt g}(\wt x,\gamma\wt y))\,dt,
\end{equation}
for any choice of lifts $\wt x,\wt y\in\wt M$ of $x,y\in M.$ To prove \propref{parametrix_estimate}, we show that as long as $d_g(x,y)$ is small enough, there is one term in the above summation which contributes the leading order asymptotics, and the rest are smaller than the claimed remainder bound. In particular, the leading term will be the one for which $\nu = 0$ and $d_{\wt g}(\wt x,\gamma\wt y) = d_g(x,y)$. The following lema demonstrates that when $x$ and $y$ are close enough together, this occurs for a unique $\gamma$, and that by choosing the lifts $\wt x,\wt y$ properly, we may assume that this occurs exactly when $\gamma = \Id.$ 

\begin{lem}\label{gammalemma}
Let $x,y\in M$ with $d_g(x,y)  \le \frac{1}{2}\inj(M,g).$ and fix a lift $\wt x\in\wt M$ of $x.$ Then, there exists a unique lift $\wt y\in \wt M$ for which $d_{\wt g}(\wt x,\wt y) = d_g(x,y)$. Additionally, if $\gamma$ is a nonidentity element of the deck transformation group, then $d_{\wt g}(\wt x,\gamma \wt y) > \frac{1}{2}\inj(M,g)$. 
\end{lem}

\begin{proof}
The existence of such a lift $\wt y$ follows immediately from the fact that $p$ is a local isometry in a ball of radius $\frac{1}{2}\inj(M,g)$ around $\wt x$. To show uniqueness, let $x,y, \wt y$ be as above, and note that any other lift of $y$ must be of the form $\gamma\wt y$ for some $\gamma \ne \Id.$ Then $d_{\wt g}(\wt y,\gamma \wt y)$ is the length of a nontrivial closed geodesic in $M$ starting and ending at $y$. Since $M$ is compact, there exists a positive minimum of the lengths of such geodesics which is independent of $y$. In fact, we have that $0 < \inj(M,g) < d_{\wt g}( \wt y,\gamma\wt y)$. Thus, by the triangle inequality, we have
\[0 < \inj(M,g) \le d_{\wt g}(\wt y,\gamma\wt y) \le d_{\wt g}(\wt y,\wt x) + d_{\wt g}(\wt x,\gamma\wt y) = d_g(x,y) + d_{\wt g}(\wt x,\gamma\wt y),\]
since $d_{\wt g}(\wt x,\wt y) = d_g(x,y)$. Using that $d_g(x,y) \le \frac{1}{2}\inj(M,g)$, we have
\[0 < d_g(x,y) \le \frac{1}{2}\inj(M,g) < d_{\wt g}(\wt x,\gamma\wt y),\]
 which demonstrates that $d_{\wt g}(\wt x,\gamma \wt y) \ne d_g(x,y)$, and also verifies the claimed lower bound on $d_{\wt g}(\wt x,\gamma \wt y)$.
\end{proof}

Next, we obtain the asymptotics of the term in \eqref{double_sum}, where $\nu = 0$ and $\gamma = \Id$. Recalling \eqref{zero_term_cosine} and \eqref{parametrix_coeffs}, this term is given by
\begin{equation}\label{leading_term_theta}
\frac{1}{\pi(2\pi)^{n}}\Theta^{-\frac{1}{2}}(x,y)\int\limits_{-\infty}^\infty\int\limits_{T_x^*M}e^{i\langle\exp_x^{-1}(y),\xi\rangle_{g}}\wh\rho(At)\frac{\sin t\lambda}{t}\cos\lp t|\xi|_{g}\rp\frac{d\xi\,dt}{\sqrt{\det g_x}},
\end{equation}
where we can use $x,y\in M$ instead of their lifts in $\wt M$ since $p$ is an isometry in a neighborhood containing $\wt x,\wt y.$ We seek to show that this term contributes the leading order behavior in \eqref{parametrix_estimate_eqn}. To accomplish this, we first study the behavior of its derivative with respect to $\lambda$, since it is more straightforward to study and will prove useful in later arguments. 

\begin{lem}\label{leading_term_mu}
Fix $\wh\rho$ as in \propref{exact_kernel_remainder}. Then for any $0 < A < 1$, we have
\begin{align}\label{leading_term_mu_eqn}
\begin{split}
&\frac{1}{\pi(2\pi)^n}\int\limits_{-\infty}^\infty\int\limits_{T_x^*M}e^{i\langle\exp_x^{-1}(y),\xi\rangle_{g}}\wh\rho(At)\cos(t\mu)\cos(t|\xi|_{g_x^{-1}})\frac{d\xi\,dt}{\sqrt{\det g_x}} \\
& \hspace{0.5in}= \frac{\mu^{n-1}}{(2\pi)^n}\int\limits_{S_x^*M}e^{i\mu\langle\exp_x^{-1}(y),\xi\rangle_{g}}\frac{d\xi}{\sqrt{\det g_x}} + R_A(x,y,\mu),
\end{split}
\end{align}
where $S_x^*M$ is the co-sphere fiber at $x$, and 
\[\sup\limits_{d_g(x,y)\le \frac{1}{2}\inj(M,g)}\left|\partial_x^\alpha\partial_y^\beta R_A(x,y,\mu)\right| = \mathcal O\lp \mu^{n-3+|\alpha|+|\beta|}\rp\]
uniformly in $A$.
 
\end{lem}

\begin{proof}
For this we argue in close analogy to the proof of \cite[Proposition 12]{Canzani2015}, although we must be cautious about the dependence on $A$ throughout the argument. Let us write the left hand side of \eqref{leading_term_mu_eqn} as
\[\frac{1}{\pi(2\pi)^n}\int\limits_{-\infty}^\infty\int\limits_{T_x^*M} e^{i\langle \exp_x^{-1}(y),\xi\rangle_g}\wh\rho(At)\cos(t\mu)\cos(t|\xi|_g)\frac{d\xi\,dt}{\sqrt{\det g_x}}.\]
Using that $\cos(a)\cos(b) = \frac{1}{2}\lp \cos(a+b) + \cos(a-b)\rp$ and $\wh\rho$ is even, we can write the above as 
\[\frac{1}{(2\pi)^{n+1}}\int\limits_{-\infty}^\infty\int\limits_{T_x^*M} e^{i\langle \exp_x^{-1}(y),\xi\rangle_g}\lp e^{it(\mu - |\xi|_g)} + e^{it(\mu +|\xi|_g)}\rp\wh\rho(At)\frac{d\xi\,dt}{\sqrt{\det g_x}}.\]
We will concern ourselves only with the term involving $e^{ it(\mu - |\xi|_g)}$, because it can be seen by repeating the following argument that the other term yields only rapidly decreasing terms in $\mu,$ due to the fact that the phase is nonstationary for $\mu > 0.$ Making the change of variables $\xi = \mu r \omega$ for $r\in\R^+$ and $\omega \in S^*_xM$, it suffices to estimate
\begin{equation}\label{polar_coords}
\frac{\mu^n}{(2\pi)^{n+1}}\int\limits_{-\infty}^\infty\int\limits_0^\infty \int\limits_{S_x^*M} e^{i\mu r\langle\exp_x^{-1}(y),\omega\rangle_g+ it\mu(1-r)}\wh\rho(At)r^{n-1}\,d\sigma_x(\omega)\,dr\,dt,
\end{equation}
where $d\sigma_x$ is the induced surface measure on $S_x^*M$. By \cite[Theorem 1.2.1]{SoggeBook2017}, we can write
\begin{equation}\label{spherical_decomp}
\int\limits_{S_x^*M}e^{i\mu r\langle\exp_x^{-1}(y),\omega\rangle_g}\,d\sigma_x(\omega) = \sum\limits_{\pm}e^{\pm i\mu r d_g(x,y)}a_{\pm}(\mu r\exp_x^{-1}(y)),
\end{equation}
where $|\partial^\alpha a_\pm(\zeta)|\le C(1 + |\zeta|)^{-\frac{n-1}{2}-|\alpha|}.$ Hence, \eqref{polar_coords} can be expressed as
\begin{equation}\label{polar_coords_decomp}
\sum\limits_{\pm}\frac{\mu^n}{(2\pi)^{n+1}}\int\limits_{-\infty}^\infty\int\limits_0^\infty e^{i\mu \psi_\pm(x,y,t,r)}a_\pm(\mu r \exp_x^{-1}(y))\wh\rho(At)r^{n-1}\,dr\,dt,
\end{equation}
where $\psi_\pm(x,y,t,r) = t(1-r)\pm rd_g(x,y)$. Motivated by the form of this phase function, we introduce a cutoff $\beta\in C_c^\infty(\R^+)$ with $\beta\equiv 1$ on small neighborhood of $r = 1$ and supported in $\lp\frac{1}{2},\frac{3}{2}\rp$ . We then have that \eqref{polar_coords_decomp} equals
\begin{equation}\label{polar_coords_cutoff}
 \sum\limits_{\pm}\frac{\mu^n}{\pi(2\pi)^{n}}\!\!\!\int\limits_{-\infty}^\infty\int\limits_0^\infty e^{i\mu \psi_\pm(x,y,t,r)}a_\pm(\mu r \exp_x^{-1}(y))\wh\rho(At)r^{n-1}\beta(r)\,dr\,dt \!+\! \mathcal O\lp\!\mu^{-N}\!\!\rp
\end{equation}
for any $N > 2n-1$, uniformly in $0 < A \le 1$ and all $x,y\in M$. To see that the remainder is $\mathcal O(\mu^{-N})$, note that if we introduce a factor of $1 - \beta(r)$ in \eqref{polar_coords_decomp}, we can integrate by parts arbitrarily many times in $t$ using the operator $\frac{1}{\mu(1-r)}\partial_t$, which is well defined on the support of $\beta$. This results in an expression of the form
\begin{equation}\label{rapid_decay}
\frac{(-1)^NA^N}{\mu^N}\int\limits_0^\infty e^{\pm i \mu r d_g(x,y)}(1-r)^{-N}r^{n-1}(1-\beta(r))\int\limits_{-\infty}^\infty e^{it(1-r)}\wh\rho^{(N)}(At)\,dt\,dr.
\end{equation}
Since $\wh\rho^{(N)}(At)$ vanishes for $|t|\ge L/A$, we have that \eqref{rapid_decay} is bounded in absolute value by a constant times $A^{N-1}\mu^{-N}$, provided that $N>2n-1$ so that the integral in the $r$ variable is absolutely convergent. Recalling that $A \le 1$ shows that the asymptotic in \eqref{polar_coords_cutoff} is uniform with respect to $A$.

Next, we seek to apply stationary phase to the first term in \eqref{polar_coords_cutoff} (see \cite[Thm 3.16]{ZworskiBook2012} and \cite{Fedoryuk1990} ). For this we set 
\[b_A^{\pm}(t,r,x,y,\mu) = a_{\pm}(\mu r\exp_x^{-1}(y))\wh\rho(At)r^{n-1}\beta(r)\]
 and note that the phase functions $\psi_{\pm}$ each have a unique critical point at  $(t_0^\pm,r_0^\pm) = (\pm d_g(x,y),1)$. Therefore, we have that the first term in \eqref{polar_coords_cutoff} equals
\begin{align}
\begin{split}\label{statphase}
&\frac{\mu^{n-1}}{(2\pi)^{n}}e^{\pm i\mu d_g(x,y)}\sum\limits_{\pm} \bigg(b_A^\pm(t_0^{\pm},r_0^{\pm},x,y,\mu) + \frac{1}{i\mu}\partial_t\partial_r b_A^\pm(t_0^\pm,r_0^\pm,x,y,\mu)\bigg) \\
& \hspace{.5in} \! +\!\frac{\mu^{n-3}}{(2\pi)^{n}}e^{\pm i\mu d_g(x,y)}\sum\limits_{\pm} F_A^\pm(x,y,\mu),
\end{split}\end{align}
where 
\[|F_A^\pm(x,y,\mu)| \le \sup\limits_{k+\ell \le 7}\sup\limits_{(t,r)\in\supp b_A^\pm}\left|\partial_t^k\partial_r^\ell b_A^\pm(t,r,x,y,\mu)\right| \le C(1 + \mu d_g(x,y))^{-\frac{n-1}{2}},\]
with $C$ independent of $A$ by our estimates on $a_\pm$, the fact that $\wh\rho$ is uniformly bounded, and the fact that $\beta$ is supported where $r \approx 1$. For $d_{g}(x,y) \le \frac{1}{2}\text{inj}(M,g)$ and $A \le 1$, we have that $\wh\rho(Ad_g(x,y)) = 1$ and $\partial_t\wh\rho(Ad_g(x,y)) = 0$, and hence we see that \eqref{statphase} is equal to 
\begin{align*}
& \frac{\mu^{n-1}}{(2\pi)^{n}}\sum\limits_{\pm}e^{\pm i \mu d_g(x,y)}a_\pm(\mu\exp_x^{-1}(y)) +  \mathcal O\lp \mu^{n-3}\rp\\
& \hskip 0.3in = \frac{\mu^{n-1}}{(2\pi)^{n}}\int\limits_{S_x^*M}e^{i\mu\langle\exp_x^{-1}(y),\omega\rangle_{g}}\,d\sigma_x(\omega) + \mathcal O(\mu^{n-3}),
\end{align*}
after recalling the decomposition \eqref{spherical_decomp}. This completes the proof in the case where we take no derivatives of the remainder. To include derivatives, we note that the dependence on $x,y$ in \eqref{polar_coords} only appears in the quantity
\[\int\limits_{S_x^*M}e^{i\mu r\langle\exp_x^{-1}(y),\omega\rangle_g}\,d\sigma_x,\]
and hence each differentiation in $x$ or $y$ yields at most one additional power of $\mu$ in the asymptotic expansion. More precisely, by the linear change of variables $\theta = g_x^{-1/2}\omega$, we have
\[\int\limits_{S_x^*M}e^{i\mu r\langle\exp_x^{-1}(y),\omega\rangle_g}\,d\sigma_x(\omega) = \int\limits_{\mathbb S^{n-1}}e^{i\mu r\langle g_x^{-1/2}\exp_x^{-1}(y),\theta\rangle_{\R^n}}\,dS(\theta),\]
where $dS$ is the surface measure on the round sphere $\mathbb S^{n-1}\subset \R^n,$ and so the dependence on $x,y$ only appears in the exponent. Therefore, applying $\partial_x^\alpha\partial_y^\beta$ yields a finite linear combination of terms of the form
\[(i\mu r)^kf(x,y)\int\limits_{S_x^*M}e^{i\mu r\langle\exp_x^{-1}(y),\omega\rangle_g}h(\omega)\,d\sigma_x(\omega)\]
for $k \le |\alpha|+|\beta|$ and some smooth, bounded functions $f,h.$ Repeating the preceding argument on each of these terms yields the desired result. 

\end{proof}

If it were not for the factor of $\Theta^{-\frac{1}{2}}$ which appears in the $\nu = 0$ term of \eqref{leading_term_theta}, we could simply integrate \eqref{leading_term_mu_eqn} with respect to $\mu$ to obtain the leading term in \eqref{parametrix_estimate} with a remainder bounded by $\mathcal O(\lambda^{n-2 +|\alpha|+|\beta|})$. The following lemma handles this difficulty at the expense of weakening the remainder bound.

\begin{lem}\label{leading_term_lambda}
For $\wh\rho$ as in \propref{exact_kernel_remainder}, there exist constants $c,C,\lambda_0 > 0$ such that if $A = \frac{1}{c\log\lambda}$, then
\begin{align}\label{leading_term}
\begin{split}
&\frac{\Theta^{-\frac{1}{2}}(x,y)}{(2\pi)^n}\int\limits_{-\infty}^\infty \int\limits_{T_x^*M}e^{i\langle\exp_x^{-1}(y),\xi\rangle_{g}}\wh\rho(At)\frac{\sin t\lambda}{t}\cos(t|\xi|_{g})\frac{d\xi\,dt}{\sqrt{\det g_x}}\\
&\hskip 0.5in = \frac{\lambda^n}{(2\pi)^n}\int\limits_{B_x^*M}e^{i\lambda\langle \exp_x^{-1}(y),\xi\rangle_{g}}\frac{d\xi}{\sqrt{\det g_x}} + R_A(x,y,\lambda),
\end{split}
\end{align}
where 
\[\sup\limits_{d_{g}(x,y) \le \frac{1}{2}\textnormal{inj}(M,g)}\left|\partial_x^\alpha\partial_y^\beta R_A(x,y,\lambda)\right| \le  \frac{C\lambda^{n-1+|\alpha|+|\beta|}}{\log\lambda}\]
for all $\lambda \ge \lambda_0.$ 
\end{lem}

\begin{proof}
We first handle the case where $|\alpha|= |\beta| = 0$. Since the differential of $\Theta^{-\frac{1}{2}}$ vanishes at $( x, x)\in M\times M$, we know that
\[
\Theta^{-\frac{1}{2}}(x,y) = 1 + d_g(x,y)^2f(x,y)
\]
for some smooth, bounded function $f$. Thus, we need only show that 
\begin{equation}\label{d_squared}
d_g(x,y)^2\!\!\!\!\int\limits_{-\infty}^\infty \int\limits_{T_x^*M}\!\!\!\! e^{i\langle\exp_x^{-1}(y),\xi\rangle_{g}}\wh\rho(At)\frac{\sin t\lambda}{t}\cos(t|\xi|_{g})\frac{d\xi\,dt}{\sqrt{\det g_x}} = \mathcal O\lp\frac{\lambda^{n-1}}{\log\lambda}\rp,
\end{equation}
since we can integrate \eqref{leading_term_mu_eqn} with respect to $\mu$ from 0 to $\lambda$ to obtain the claimed leading order term with an $\mathcal O(\lambda^{n-2})$ error. Observe that 
\[d_g(x,y)^2 e^{i\langle\exp_x^{-1}(y),\xi\rangle_{g}} = \frac{1}{i}\langle \exp_x^{-1}(y),\nabla_{\xi} e^{i\langle\exp_x^{-1}(y),\xi\rangle_{g}}\rangle_{g}\]
where $\nabla_\xi$ denotes the induced gradient on the cotangent fiber $T_x^*M$. Thus, we may integrate by parts in $\xi$ on the left-hand side of \eqref{d_squared} to obtain
\begin{equation}\label{ibp_xi}
\frac{1}{i}\int\limits_{-\infty}^\infty \int\limits_{T_x^*M} e^{i\langle\exp_x^{-1}(y),\xi\rangle_g}\wh\rho(At)\left\langle \exp_x^{-1}(y),\frac{\xi}{|\xi|_g}\right\rangle_{\!\!\!g}\sin(t\lambda)\sin(t|\xi|)\frac{d\xi\,dt}{\sqrt{\det g_x}}.
\end{equation}
 Since $\langle\exp_x^{-1}(y),\xi/|\xi|\rangle$ can be written as $d_g(x,y)$ times a bounded function of $x,$ $y,$ and $\xi/|\xi|$, and since $\sin(a) \sin(b) = \frac{1}{2}\lp\cos(a-b) - \cos(a+b)\rp$, we may repeat arguments from the proof of \propref{leading_term_mu} to see that \eqref{ibp_xi} is bounded by a constant times 
\begin{equation}\label{stat_phase_bound}
d_g(x,y)\lambda^{n-1}(1 + \lambda d_g(x,y))^{-\frac{n-1}{2}}.
\end{equation}
In the regime where $d_g(x,y) \le \frac{1}{\log\lambda}$, \eqref{stat_phase_bound} is clearly bounded by $\mathcal O\lp\lambda^{n-1}/\log\lambda\rp$. If $\frac{1}{\log\lambda}\le d_g(x,y) \le \frac{1}{2}\inj(M,g)$, then we have that
\[d_g(x,y)\lambda^{n-1}(1 + \lambda d_g(x,y))^{-\frac{n-1}{2}} \le
C\lambda^{\frac{n-1}{2}}(\log\lambda)^{\frac{n-1}{2}} \le \frac{C\lambda^{n-1}}{\log\lambda},\]
since $n \ge 2.$ This completes the proof in the case of no $x,y$ derivatives. 

To include $\partial_x^\alpha\partial_y^\beta$, we must consider a few cases. As discussed in the proof of \propref{leading_term_mu}, each derivative which falls on the integral in the left-hand side of \eqref{leading_term} yields one additional power of $\lambda$ in the asymptotic expansion. If every derivative falls on the integral, then we have precisely the claimed leading order term plus a remainder on the order of $\lambda^{n-1 + |\alpha|+|\beta|}/\log\lambda$ by combining \propref{leading_term_mu}, an integration from 0 to $\lambda$ in $\mu$, and a repetition of the above argument. Alternatively, if two or more of the derivatives fall on the $\Theta^{-\frac{1}{2}}$ factor, then \propref{leading_term_mu} shows that the contributions from the integral itself are at most $\lambda^{n - 2 + |\alpha|+|\beta|}$, and then we simply use that all derivatives of $\Theta^{-\frac{1}{2}}$ are bounded when $x,y$ are restricted to a compact set. The only remaining case is the scenario in which exactly one derivative falls on the $\Theta^{-\frac{1}{2}}$ factor. Here we must use the fact that the differential of $\Theta^{-\frac{1}{2}}(x,y)$ vanishes on the diagonal in $M\times M,$ and hence both $\partial_{x_j}\lp\Theta^{-\frac{1}{2}}(x,y)\rp$ and $\partial_{y_j}\lp\Theta^{-\frac{1}{2}}(x,y)\rp$ are $\mathcal O\lp d_g(x,y)\rp$ for any $j$. Combining this with previous arguments, we have that if $\alpha'$ is a multiindex of length $|\alpha|-1$, then 
\begin{align}\label{theta_one_d}
\begin{split}
\left|\partial_{x_j}(\Theta^{-\frac{1}{2}}(x,y))\partial_x^{\alpha'}\partial_y^{\beta} \int\limits_{-\infty}^\infty \int\limits_{T_x^*M}\!\!\!\! e^{i\langle\exp_x^{-1}(y),\xi\rangle_{g}}\wh\rho(At)\frac{\sin t\lambda}{t}\cos(t|\xi|_{g})\frac{d\xi\,dt}{\sqrt{\det g_x}}\right| & \\
& \hskip -2in \le C d_g(x,y)\lambda^{n-1 +|\alpha|+ |\beta|}(1 + \lambda d_g(x,y))^{-\frac{n-1}{2}}.
\end{split}
\end{align}
Arguing as before, we see that the right-hand side of \eqref{theta_one_d} is bounded by $\mathcal O\lp \lambda^{n-1 + |\alpha|+|\beta|}/\log\lambda\rp$ by considering the regions where $d_g(x,y)\le \frac{1}{\log\lambda}$ and $d_g(x,y) \ge \frac{1}{\log\lambda}$ separately. An analogous estimate holds with $\partial_{y_j}$ replacing $\partial_{x_j}.$ 
\end{proof}

Next, we estimate the terms in \eqref{double_sum} with $\gamma = \Id$ and $\nu \ge 1.$

\begin{lem}\label{near_diagonal_terms}
For $\nu = 1,2,\dotsc,$ and any $\delta > 0$, there exist constants $c,C_\nu,\lambda_0> 0$ such that if $A = \frac{1}{c\log\lambda}$,
\begin{align}\label{near_diagonal}
\begin{split}
& \sup\limits_{d_g(x,y)\le \frac{1}{2}\inj(M,g)}\left|\partial_x^\alpha\partial_y^\beta\lp u_\nu(x,y)\int\limits_{-\infty}^\infty \frac{\sin t\lambda}{t}\wh\rho(At)\partial_t W_\nu(t,d_{ g}(x,y))\,dt\rp\right|  \le C_\nu\max\{\lambda^{n-\nu-1+|\alpha|+|\beta|},\lambda^\delta\}
\end{split}
\end{align}
for all $\lambda \ge \lambda_0.$
\end{lem}

\begin{proof}
Since $u_\nu$ is $C^\infty$ and $x,y$ are restricted to a compact set, derivatives of $u_\nu$ are uniformly bounded by some constant depending only on $\nu$ and the order of differentiation. Next, we recall that by \eqref{linear_comb}, it suffices to estimate
\begin{equation}\label{near_diagonal_suffices}
\partial_x^\alpha\partial_y^\beta\lp\int\limits_{-\infty}^\infty \int\limits_{T_x^*M} e^{i\langle\exp_x^{-1}(y),\xi\rangle_g + it(\lambda \pm |\xi|_g)}\wh\rho(At) \chi(|\xi|_g)t^{j}|\xi|_g^{-\nu-k}\frac{d\xi\,dt}{\sqrt{\det g_x}}\rp
\end{equation}
for any nonnegative integers $j,k$ with $j + k = \nu -1,$ where $\chi \equiv 0$ on $[-1,1]$ and $\chi \equiv 1$ outside $[-2,2].$ To see that this is sufficient, we must show that the error term in \eqref{linear_comb} contributes only negligible terms to the asymptotics in $\lambda$. Let $f_\nu:\R \times T_x^*M\to\C$ be a smooth, tempered function. Then 
\[\left|\int\limits_{-\infty}^\infty \wh\rho(At)e^{it\lambda}f_\nu(t,\exp_x^{-1}(y))\,dt\right| \le C\int\limits_{-\infty}^\infty |\wh\rho(At)|(1 + d_g(x,y))^p(1 + |t|)^q\,dt \]
for some $p,q\ge 0$ since $f_\nu$ is tempered. Since $d_g(x,y)$ is bounded, we have that the above is dominated by a constant times
\[\int\limits_{-\infty}^\infty|\wh\rho(At)|(1 + |t|)^q\,dt \le C \int\limits_{-1/A}^{1/A}(1 + |t|)^q\,dt,\]
which is certainly bounded by $C_1e^{C_2/A}$ for some $C_1,C_2 > 0.$ For $A = \frac{1}{c\log\lambda}$ with $c$ sufficiently small, we then have that this contributes at most $\lambda^\delta$ with $\delta> 0$ small. The same is true if we introduce derivatives of $f$ with respect to $x,y$. Therefore, the proof will be complete once we show that \eqref{near_diagonal_suffices} satisfies the correct bound.

Changing to polar coordinates via $\xi = r \omega$ , we have that \eqref{near_diagonal_suffices} equals 
\begin{equation}\label{near_diag_polar}
\int\limits_{-\infty}^\infty \int\limits_0^\infty \int\limits_{S_x^*M}\!\!\!\! e^{i r\langle\exp_x^{-1}(y),\omega\rangle_g + it(\lambda \pm r)}\wh\rho(At)\chi( r)t^{j}r^{n-1 - \nu - k}\,d\sigma_x(\omega)\,dr\,dt.
\end{equation}
Noting that $t^{j}e^{\pm it r} = (\pm\frac{1}{i}\partial_r)^{j}e^{\pm i t r}$, we may integrate by parts $j $ times in $r$. This is justified in the sense of distributions, even if the integral in $r$ is not absolutely convergent. If any derivatives fall on the $\chi(r)$ factor, the resulting integrand will be compactly supported in $r$, and so combining the preceding argument with the discussion prior to \eqref{linear_comb}, we see that modulo an $\mathcal O\lp\lambda^\delta\rp$ error, \eqref{near_diag_polar} can be written as a finite linear combination of terms of the form 
\begin{align}\label{near_diag_ibp}
\begin{split}
& \int\limits_{-\infty}^\infty \int\limits_0^\infty \!\!\int\limits_{S_x^*M}\!\!\!\! e^{ir\langle\exp_x^{-1}(y),\omega\rangle_g+ it(\lambda \pm r)}\wh\rho(At)\chi(r)\times \\
&\hskip 0.5in \times \langle \exp_x^{-1}(y),\omega\rangle_g^\ell r^{n-1 -\nu-k-j+\ell}\,d\sigma_x(\omega)\,dr\,dt
\end{split}
\end{align}
for $0 \le \ell \le j$. Rescaling via $r\mapsto \lambda r$, and recalling that $j + k = \nu -1,$ we obtain
\begin{align}\label{chi_lambda_r}
\begin{split}
& \lambda^{n-2\nu + \ell + 1}\!\!\!\! \int\limits_{-\infty}^\infty \int\limits_0^\infty \int\limits_{S_x^*M}\!\!\!\! e^{i\lambda r\langle\exp_x^{-1}(y),\omega\rangle_g+ it\lambda (1 \pm r)}\wh\rho(At)\chi(\lambda r) \langle \exp_x^{-1}(y),\omega\rangle_g^\ell r^{n-2\nu+\ell}\,d\sigma_x(\omega)\,dr\,dt.
\end{split}
\end{align}
We now wish to apply the stationary phase argument from the proof of \lemref{leading_term_mu}. One potential difficulty that arises is that the cutoff $\chi$ is scaled by $\lambda$, and so it appears that in the corresponding analogue of \eqref{statphase}, one may have extra factors of $\lambda$ which appear due to differentiating $\chi(\lambda r)$ with respect to $r$. However, we recall that the $\beta$ from the proof of \lemref{leading_term_mu} was supported in $\lp\frac{1}{2},\frac{3}{2}\rp$, and $\chi(\lambda r)$ is identically 1 for $r \ge \frac{2}{\lambda}$. Thus, $\partial_r^k\chi(\lambda r)$ is zero for $r > \frac{2}{\lambda}$. So, for large enough $\lambda$, the derivatives of $\chi$ will vanish on the support of $\beta$, and the problem is avoided. We may therefore apply the exact same argument as in the proof of \lemref{leading_term_mu} to see that \eqref{chi_lambda_r} is bounded by $\lambda^{n-2\nu+\ell}(1+\lambda d_g(x,y))^{-\frac{n-1}{2}}$. Since $\ell \le j \le \nu -1$, we have that $n-2\nu + \ell \le n - \nu - 1$, giving the exponent we claimed in \lemref{near_diagonal_terms}. As discussed previously, adding derivatives $\partial_x^\alpha\partial_y^\beta$ simply adds at most $|\alpha|+|\beta|$ additional powers of $\lambda$ from the $e^{i\lambda r\langle\exp_x^{-1}(y),\omega\rangle_g}$ factor, and so the proof is complete. 
\end{proof}

Finally, we must control the terms in \eqref{double_sum} for which $\gamma \ne \Id$. Here we must work in the universal cover and take advantage of the fact that the lifts $\wt x$ and $\gamma \wt y$ are bounded away from each other. This allows us to improve our estimates on the corresponding terms by a power of $\frac{n-1}{2}$ by exploiting the factors of $(1 + \lambda d_{\wt g}(\wt x,\gamma\wt y))^{-\frac{n-1}{2}}$ which appear when we apply stationary phase. 

\begin{lem}\label{off_diagonal}
Given any $\delta> 0$, there exist constants $c,C_\nu,\lambda_0 > 0$ such that if $A = \frac{1}{c\log\lambda}$ and $\wt x,\wt y\in \wt M$ are such that $d_{\wt g}(\wt x,\wt y)\ge \ve$ for some $\ve > 0$, then
\[\left|\partial_{\wt x}^\alpha \partial_{\wt y}^\beta\!\lp \!\! u_\nu(\wt x,\wt y)\!\!\!\int\limits_{-\infty}^\infty\!\!\!\frac{\sin t\lambda}{t}\wh\rho(At)\partial_tW_\nu(t,d_{\wt g}(\wt x,\wt y))\,dt\rp  \right|\!\le\! C_\nu\!\max\{\lambda^{\frac{n-1}{2}-\nu +|\alpha|+|\beta|+\delta},\lambda^\delta\}\]
for $\lambda \ge \lambda_0$. 
\end{lem}

\begin{proof}
The argument proceeds in much the same way as the proof of \lemref{near_diagonal_terms}, although we must be cautious about the fact that the $\wt x,\wt y$ need not be restricted to a fixed compact set. However, we may recall that $\partial_t W_\nu$ vanishes when $d_{\wt g}(\wt x,\wt y) > |t|$ and that $\wh\rho(At)$ vanishes when $|t| \ge L/A.$ Hence, we may assume that $d_{\wt g}(\wt x,\wt y) \le \frac{L}{A}.$ By \lemref{u_nu_bound}, we have that under this restriction on $d_{\wt g}(\wt x,\wt y)$,
\begin{equation}\label{u_nu_epsilon}
|P_{\wt x}Q_{\wt y}u_\nu(\wt x,\wt y)|\le C_1e^{C_2d_{\wt g}(\wt x,\wt y)}\le C_1 e^{C_2/A} = C_1 \lambda^{C_2c}
\end{equation}
if $A = \frac{1}{c\log\lambda}$. We can then choose $c$ small enough so that \eqref{u_nu_epsilon} is bounded by $\mathcal O\lp \lambda^{\delta/2}\rp$. Note that this choice of $c$ depends only on $\delta$, $\nu,$ and the order of differentiation. Therefore, it suffices to prove that
\begin{align}\label{off_diag_suffices}
\begin{split}
& \sup\limits_{\ve\le d_{\wt g}(\wt x,\wt y)\le \frac{L}{A}}\left|P_{\wt x}Q_{\wt y}\lp\int\limits_{-\infty}^\infty\frac{\sin t\lambda}{t}\wh\rho(At)\partial_t W_\nu(t,d_{\wt g}(\wt x,\wt y))\,dt\rp\right|\le C_\nu\max\{\lambda^{\frac{n-1}{2}-\nu + |\alpha|+|\beta| + \frac{\delta}{2}},\lambda^{\frac{\delta}{2}}\}.
\end{split}
\end{align}
We argue as in the beginning of the proof of \lemref{near_diagonal_terms} to show that it is in fact enough to estimate 
\begin{equation}\label{off_diag_linear_comb}
P_{\wt x}Q_{\wt y}\lp\int\limits_{-\infty}^\infty \int\limits_{T_{\wt x}^*\wt M} e^{i\langle\exp_{\wt x}^{-1}(\wt y),\xi\rangle_g + it(\lambda \pm |\xi|_g)}\wh\rho(At) \chi(|\xi|_g)t^{j}|\xi|_g^{-\nu-k}\frac{d\xi\,dt}{\sqrt{|\wt g_{\wt x}|}}\rp.
\end{equation}
To reduce to this case, we must show that the smooth, tempered error $f_\nu(t,\exp_{\wt x}^{-1}(\wt y))$ in \eqref{linear_comb} introduces a negligible contribution to the growth in $\lambda$ as before. The new concern is that the $\wt x$ and $\wt y$ are not restricted to a compact set, and so if we differentiate $f_\nu(t,\exp_{\wt x}^{-1}(\wt y))$ with respect to $\wt x$ or $\wt y$, we must be able to control the derivatives of $\exp_{\wt x}^{-1}(\wt y)$ which appear due to the chain rule. It is here that we must apply \lemref{exp_map}, which states that all derivatives of the inverse exponential map are bounded at most exponentially in $d_{\wt  g}(\wt x,\gamma\wt y)$. Combining this with the fact that $f_\nu$ is a tempered function, we have that 
\[\left|\partial_{\wt x}^\alpha\partial_{\wt y}^\beta f_{\nu}(t,\exp_{\wt x}^{-1}(\wt y))\right| \le C_1 e^{C_2 d_{\wt g}(\wt x,\wt y)}(1 + |t|)^p \]
for some constants $C_1,C_2,p> 0$ which depend only on $\nu$ and the order of differentiation. Hence, for $|d_{\wt g}(\wt x,\wt y)|\le \frac{L}{A}$, we have
\[\left|\int\limits_{-\infty}^\infty \wh\rho(At)e^{it\lambda}\partial_{\wt x}^\alpha\partial_{\wt y}^\beta f_\nu(t,\exp_{\wt x}^{-1}(\wt y))\,dt\right| \le C_1 e^{C_2/A}\int\limits_{-L/A}^{L/A}(1 + |t|)^{p}\,dt \le C_1 e^{C_2/A}\]
after potentially increasing $C_1$ and $C_2.$ As discussed previously, we can then choose $c$ small enough so that the above is bounded by $\mathcal O\lp\lambda^{\delta/2}\rp$. Therefore, we only need to show that \eqref{off_diag_linear_comb} is bounded by $\mathcal O\lp \lambda^{\frac{n-1}{2}-\nu +|\alpha|+|\beta|+\frac{\delta}{2}}\rp$ for $d_{\wt g}(\wt x,\wt y)\le \frac{L}{A}$. For the case where we take no derivatives, we may repeat the proof of \lemref{near_diagonal_terms} to obtain a linear combination of terms, each with a bound of the form $C_\nu\lambda^{n-2\nu + \ell}(1 + \lambda d_{\wt g}(\wt x,\wt y))^{-\frac{n-1}{2}}$ for $0 \le \ell \le \nu-1$. However, in this case, the distance between $\wt x,\wt y $ is bounded below by $\frac{1}{2}\inj(M,g),$ and so the previously mentioned terms are all bounded by $C_\nu\lambda^{\frac{n-1}{2}-\nu}$ uniformly in $\wt x,\wt y$ under our conditions on $\ell.$ In order to include derivatives, we may again repeat previous arguments to show that we obtain at most $|\alpha|+|\beta|$ extra powers of $\lambda,$ but we must take into account the possibility that we obtain a factor involving derivatives of $\exp_{\wt x}^{-1}(\wt y)$. In such a case, we simply apply \lemref{exp_map} and previous discussions to see that this contributes at worst an extra $\mathcal O\lp\lambda^{\delta/2}\rp$ factor.

\end{proof}

In light of the three preceding lemmas, the proof of \propref{parametrix_estimate} is nearly complete. The final step is to recall that by \eqref{gamma_count} and finite speed of propagation, the number of nonzero terms in \eqref{double_sum} with $\gamma \ne \Id$ is bounded by a constant times $e^{C/A}$, and hence is bounded by $\lambda^\delta$ with $\delta$ small if we choose $A = \frac{1}{c\log\lambda}$ with $c$ small enough. Therefore, by \lemref{off_diagonal} and the triangle inequality we have that for any $P,Q\in \textnormal{Diff}(\wt M)$ of orders $|\alpha|$ and $|\beta|$, respectively,
\begin{align*}
& \sum\limits_{\gamma \ne \Id}\sum\limits_{\nu = 0}^N\left|P_{\wt x}Q_{\wt y}\!\lp\!\frac{1}{\pi}u_\nu(\wt x,\gamma\wt y)\!\!\int\limits_{-\infty}^\infty\!\! \wh\rho(At)\frac{\sin t\lambda}{t}\partial_t W_\nu(t,d_{\wt g}(\wt x,\gamma\wt y))\,dt\!\!\rp\right| \le C\max\{\lambda^{\frac{n-1}{2}-\nu+|\alpha|+|\beta| + 2\delta},\lambda^{2\delta}\}
\end{align*}
for some $C>0.$ 
Combining this with Lemmas \ref{leading_term_lambda} and \ref{near_diagonal_terms}, the proof of \propref{parametrix_estimate} is complete. In combination with Propositions \ref{exact_kernel_remainder} and \ref{parametrix_error}, we can see that this completes the proof of \thmref{logthm}.

\begin{rmk}[Proof of \thmref{logthm_2}]\label{safarov_remark}\textnormal{
We note that throughout the entire proof of \thmref{logthm}, the only reason we needed $d_g(x,y)$ to be small was so that we could uniquely determine which term in the parametrix expansion gives the leading order behavior, which allows us to write the asymptotic \eqref{standardweyl}. However, if one assumes that $d_g(x,y)\ge \ve$ for some $\ve > 0$, then the only issues that arise are that there may be a finite collection of $\gamma\in\Gamma$ for which $d_{\wt g}(\wt x,\gamma\wt y) = d_g(x,y)$, and that $\exp_{x}^{-1}(y)$ is no longer necessarily well-defined. However, in such a case, $\exp_{\wt x}^{-1}(\wt y)$ still makes sense on $\wt M$, and we have that $d_{\wt g}(\wt x,\gamma\wt y)$ is bounded below by a positive constant for \emph{every} $\gamma,$ since it is impossible for the distance between any two lifts $\wt x,\wt y$ to be smaller than $d_g(x,y)$. This is due to the fact that geodesics on $\wt M$ project to geodesics on $M$ via the covering map. Hence, one could apply \lemref{off_diagonal} to all the terms in the parametrix to obtain that the integral on the left-hand side of \eqref{parametrix_estimate_eqn} satisfies
\[\left|\frac{1}{\pi}\int\limits_{-\infty}^\infty \frac{\sin t\lambda}{t}\wh\rho(At)\partial_x^\alpha\partial_y^\beta K_N(t,x,y)\,dt\right| \le C\lambda^{\frac{n-1}{2}+|\alpha|+|\beta| + \delta}\]
for some small $\delta> 0.$ Since this bound is smaller than $\mathcal O\lp\frac{\lambda^{n-1+|\alpha|+|\beta|}}{\log\lambda}\rp$, we can combine this with Propositions \ref{exact_kernel_remainder} and \ref{parametrix_error} to see that we obtain an upper bound of the form 
\[\sup\limits_{d_g(x,y)\ge \ve}\left|\partial_x^\alpha\partial_y^\beta E_\lambda(x,y)\right|\le \frac{C\lambda^{n-1+|\alpha|+|\beta|}}{\log\lambda}\]
for any $\ve>0$, which is exactly the content of \thmref{logthm_2}.}
\end{rmk}




\section{Proof of \thmref{spectral_cluster_thm}}
\label{cluster_thm_proof}
In this section, we show that Theorem \ref{spectral_cluster_thm} follows from \thmref{logthm} in a straightforward manner. 

\begin{proof}[Proof of Theorem \ref{spectral_cluster_thm}]
Recalling the definition of $E_{_{\!(\lambda,\lambda+1]}}(x,y)$ in \eqref{spectral_cluster_kernel}, \thmref{logthm} implies that
\begin{equation}\label{cluster_leadingterm}
E_{_{\!(\lambda,\lambda+1]}}(x,y)= \frac{1}{(2\pi)^n}\!\!\!\!\!\int\limits_{\lambda<|\xi|\le \lambda+1}\!\!\!e^{i\langle\exp_x^{-1}(y),\xi\rangle_{g}}\frac{d\xi}{\sqrt{\det g_x}} + R_{(\lambda,\lambda+1]}(x,y),
\end{equation}
where $R_{_{\!(\lambda,\lambda+1]}}(x,y) = R_{\lambda+1}(x,y) - R_\lambda(x,y)$ satisfies
\[\sup\limits_{d_g(x,y)\le \ve}\left|\partial_x^\alpha\partial_y^\beta R_{_{\!(\lambda,\lambda+1]}}(x,y)\right|\le \mathcal O\lp\frac{\lambda^{n-1+|\alpha|+|\beta|}}{\log\lambda}\rp.\]
 We then define 
\[F(\tau) = \frac{1}{(2\pi)^n}\int\limits_0^\tau\int\limits_{S_x^*M}e^{i r\langle\exp_x^{-1}(y),\omega\rangle_{g}}r^{n-1}\,dr \,d\sigma_x(\omega),\]
where $d\sigma_x$ denotes the induced measure on $S_x^*M$, so that the first term on the right-hand side of \eqref{cluster_leadingterm} equals $F(\lambda+1) - F(\lambda)$. By Taylor's theorem, we see that
\[F(\lambda+1) - F(\lambda) = F'(\lambda) + \frac{1}{2}F''(\tau),\]
for some $\tau\in (\lambda,\lambda+1)$. Since 
\[F'(\lambda) = \frac{\lambda^{n-1}}{(2\pi)^n}\int\limits_{S_x^*M}e^{i\lambda\langle\exp_x^{-1}(y),\omega\rangle_{g}}\,d\sigma_x(\omega) = \frac{\lambda^{n-1}}{(2\pi)^{n/2}}\frac{J_{\frac{n-2}{2}}(\lambda d_g(x,y))}{(\lambda d_g(x,y))^{\frac{n-2}{2}}},\]
it suffices to show that $F''(\tau)$ is smaller than the remainder bound claimed in \thmref{spectral_cluster_thm}. By direct computation, we see that 
\begin{align}\label{F''}
\begin{split}
F''(\tau) & = (n-1)\tau^{n-2}\!\!\!\!\int\limits_{S_x^*M}\!\!\!\!e^{i\tau\langle\exp_x^{-1}(y),\omega\rangle_{g}}\,d\sigma_x(\omega)\\
& \hskip 0.5in  + \tau^{n-1}\!\!\!\!\int\limits_{S_x^*M}\!\!\!i\langle\exp_x^{-1}(y),\omega\rangle_{g}e^{i\tau\langle\exp_x^{-1}(y),\omega\rangle_{g}}\,d\sigma_x(\omega).
\end{split}
\end{align}
For the first term, we can simply use that the integral is a uniformly bounded function of $\tau$ to obtain a bound of size $\mathcal O\lp\lambda^{n-2}\rp$ for $\tau\in(\lambda,\lambda+1)$, which is certainly smaller than $\mathcal O\lp\lambda^{n-1}/\log\lambda\rp$. To estimate the second term in \eqref{F''}, we can simply repeat arguments from the proof of \lemref{leading_term_lambda} to see that it is bounded by a constant times
\[d_g(x,y)\lambda^{n-1}(1 + \lambda d_g(x,y))^{-\frac{n-1}{2}},\]
for our range of $\tau.$ By considering the regions where $d_g(x,y) \le \frac{1}{\log\lambda}$ and $d_g(x,y)\ge \frac{1}{\log\lambda}$ separately as before, we obtain that the above is indeed bounded by $\mathcal O\lp\lambda^{n-1}/\log\lambda\rp$. As discussed in previous arguments, we may include derivatives in $x,y$ by simply noting that each differentiation yields at most one additional power of $\tau$ in \eqref{F''}. Thus, the proof of \thmref{spectral_cluster_thm} is complete.

\end{proof}




\appendix
\section{Localized Summations and Integrals}
\label{appendix}
In this appendix we prove a technical estimate on summations of the form 
\[\sum\limits_{k=1}^\infty(1 + |\lambda - k|)^{-N}k^p,\]
where $N$ is large, so that the summand is localized to where $k\approx \lambda$. The estimate was used in the proof of \propref{exact_kernel_remainder}, but the proof of the estimate itself is not particularly instructive, so we present the argument here. In order to prove the estimate for sums, it is convenient to first prove an estimate for integrals with a similar form. The version for sums then follows from a comparison argument. 

\begin{lem}\label{localized_integrals}
Let $p\in\R$ . Then there exists an integer $N_0>0$ and a constant $C > 0$ such that
\begin{equation}\label{localized_one_dimension}
\int\limits_1^\infty (1 + \big|\lambda - r\big|)^{-N}(1+r)^p\,dr \le C \max\{\lambda^p,1\}
\end{equation}
for all $\lambda \ge 1$ and for all $N \ge N_0$. In addition, if $p \ge 0$, then the above estimate holds for the integral over $0 \le r <\infty$.
\end{lem}

\begin{proof}
 First note that it is natural to consider the integrals over $[1,\lambda)$ and $(\lambda,\infty)$ separately. Observe that 
\begin{equation}\label{1_to_lambda}
\int\limits_0^\lambda(1 +\lambda - r)^{-N}(1+r)^p \,dr \le C\max\{\lambda^p,1\} \int\limits_0^\lambda (1 + \lambda - r)^{-N}\,dr.
\end{equation}
Then, by the change of variables $y = 1 + \lambda - r$, we get that 
\[\int\limits_0^\lambda (1 + \lambda-r)^{-N}\,dr  = \int\limits_1^{1 + \lambda}y^{-N}\,dy \le \int\limits_1^\infty y^{-N}\,dy\]
and $\int_1^\infty y^{-N}\,dy$ is bounded by a uniform constant for all $N\ge 2$.
Combining the above with \eqref{1_to_lambda}, we have
\[\int\limits_0^\lambda (1 + \lambda - r)^{-N}(1+r)^p\,dr \le C\lp \max\{\lambda^p,1\}\rp.\]

Now, consider the integral over $[\lambda,\infty)$. Here, we make the analogous change of variables ${y = 1 + r -\lambda}$ to obtain
\begin{equation}\label{lambda_to_infinity}
\int\limits_\lambda^{\infty}(1+r-\lambda)^{-N}(1+r)^p\,dr = \int\limits_1^{\infty}y^{-N}(\lambda + y)^p\,dy.
\end{equation}
If $p \le 0$, then we can bound the integrand by $y^{-N}$ since $\lambda + y \ge 1$, and we immediately see that the right-hand side of \eqref{lambda_to_infinity} is bounded by a constant. In the case where $p > 0,$ we have that
\[\int\limits_1^{\infty}y^{-N}(\lambda + y)^p\,dy \le \lambda^p \int\limits_1^\infty y^{-N}\,dy\le C\lambda^p\]
for some $C > 0,$ which completes the proof.
\end{proof}

By a simple comparison argument, one can prove the analogous result for sums.

\begin{cor}\label{localized_sums}
If $p\ge 0$, then there exist $N_0,C,\lambda_0>0$ large enough so that
\begin{equation}
\sum\limits_{k=0}^\infty (1 + \big|\lambda - k\big|)^{-N}k^p \le C\lambda^p
\end{equation}
for all $\lambda \ge \lambda_0$ and all $N\ge N_0.$
\end{cor}

\section{Geometric Estimates}
\label{geometric_estimates}

In this section, we prove growth estimates on derivatives of the Hadamard coefficients $u_\nu$, the inverse exponential map $(\wt x,\wt y)\mapsto \exp_{\wt x}^{-1}(\wt y)$, and the squared-distance function $d_{\wt g}(\wt x,\wt y)$ on the universal cover of a manifold without conjugate points. These estimates were used repeatedly in Sections \ref{hadamard} and \ref{asymptotics} in order to include derivatives in the statement of \thmref{logthm}. As in \thmref{logthm}, let $(M,g)$ be a smooth, compact Riemannian manifold without boundary and with no conjugate points. Denote by $(\wt M,\wt g)$ its universal cover, which is diffeomorphic to $\R^n$ by the Hadamard-Cartan Theorem.

\begin{prop}\label{u_nu_bound}
 Let $P,Q$ be elements of $ \textnormal{Diff}(\wt M)$, the algebra of $C^\infty$-bounded differential operators on $\wt M$, defined in the sense of \cite[Appendix A.1]{Shubin1992}. Then, we have that 
\begin{equation}
|P_{\wt x}Q_{\wt y}u_\nu(\wt x,\wt y)| \le C_1 e^{C_2 d_{\wt g}(\wt x,\wt y)}
\end{equation}
for some $C_1,C_2> 0$ which may depend on $\nu,\,P,$ and $Q.$ Here the subscripts on $P$ and $Q$ indicate the variable of differentiation.
\end{prop}

\begin{proof}
By induction and \eqref{parametrix_coeffs}, it suffices to prove the bound for derivatives of the first Hadamard coefficient, $u_0(\wt x,\wt y) = \Theta(\wt x,\wt y)^{-\frac{1}{2}}$. Recalling the definition of the $\Theta$-function, we have
\[\Theta(\wt x,\wt y)= |\text{det}\, \lp D \exp_{\wt x}\rp_{\exp_{\wt x}^{-1}(\wt y)}|.\]
By \cite[Lemma 3]{Bonthonneau2017} we have that this function is uniformly bounded below by a constant times $d_{\wt g}(\wt x,\wt y)^{1-n}$ when $d_{\wt g}(\wt x,\wt y)$ is bounded away from zero, and hence $\Theta^{-\frac{1}{2}}$ is bounded above by $Cd_{\wt g}(\wt x,\wt y)^{\frac{n-1}{2}}$ off the diagonal. Hence, by the chain rule, it suffices to estimate the derivatives of $\Theta$ in order to obtain the bound on $u_0.$ Fix $\wt x_0,\wt y_0\in\wt M$ and assume without loss of generality that $d_{\wt g}(\wt x_0,\wt y_0)\ge 1$. Let $U,V$ be small open neighborhoods of $0$ in $\R^n$ and let $\varphi:U\to \wt M$ and $\psi:V\to\wt M$ be geodesic normal coordinate charts near $\wt x_0 $ and $\wt y_0$, respectively, with $\varphi(0) = \wt x_0$ and $\psi(0) = \wt y_0$. That is, the maps $w_j\mapsto \varphi(0,\dotsc,w_j,\dotsc,0)$ and $z_j\mapsto \psi(0,\dotsc,z_j,\dotsc,0)$ are geodesics in $\wt M$ passing through $\wt x_0$ and $\wt y_0,$ respectively. Then, since $P,Q\in \text{Diff}(\wt M)$, they can be expressed in the $w$ and $z$ coordinates as
\[P = \sum\limits_{|\alpha|\le j}p_\alpha(w)\partial_{w}^\alpha\hskip 0.3in \text{ and }\hskip 0.3in Q = \sum\limits_{|\beta|\le k}q_\beta (z)\partial_{z}^\beta\]
for some $j,k\ge 0$, where the coefficient functions $p_\alpha$, $q_\beta$ are uniformly bounded in the $C^\infty$ topology on any canonical coordinate patch of fixed radius \cite[Appendix A.1]{Shubin1992}. Therefore, it suffices to estimate iterated applications of $\partial_{w}$ and $\partial_z$ to $\Theta$ in these coordinates. To accomplish this, we will consider a $2n$-dimensional variation through geodesics, motivated by the argument in \cite[\S 3]{Blair2018}. Set $\rho_0 = d_{\wt g}(\wt x_0,\wt y_0)$ and define the map $F:U\times V\times \R\to\wt M$ by 
\[F(w,z,t) = \exp_{\varphi(w)}\lp\frac{t}{\rho_0} \exp_{\varphi(w)}^{-1}(\psi(z))\rp,\]
which is a $2n$-dimensional variation through geodesics in the sense that the map $t\mapsto F(w,z,t)$ is a geodesic parametrized with speed $d_{\wt g}(\varphi(w),\psi(z))/\rho_0$ for each fixed $w,z$. Observe that in the $w,z$ coordinates $(D\exp_{\wt x_0})_{\exp_{\wt x_0}^{-1}(\wt y_0)}$ is a matrix whose columns are given by $\partial_{z_j}F(0,0,\rho_0),$ and hence it suffices to show that the lengths of the vector fields $\partial_{z_j}F\big|_{t=\rho_0}$ and their covariant derivatives in the $w,z$ coordinate directions are bounded exponentially in $\rho_0$. Since $F$ is a variation through geodesics, we have that for each fixed $j$, $\partial_{z_j}F$ is a Jacobi field along the geodesic $t\mapsto F(w,z,t)$ (c.f. \cite{LeeBook2018}). To estimate the covariant derivatives of these Jacobi fields, one may argue in close analogy to the proof of \cite[Lemma 3.3]{Blair2018} with some small modifications. Since the proof is so similar, we will not reproduce it in its entirety; we will instead sketch the argument and point out the places where the differences occur. One notable difference is that we use \cite[Lemma 4]{Bonthonneau2017} to obtain certain lower bounds without relying on the nonpositive curvature assumption of \cite[Lemma 3.3]{Blair2018}. 

 The precise estimate we seek to prove is as follows. For any integer $k\ge 0$, let $\mathcal D^k$ denote some iterated combination of elements of the set 
 \[\mathscr D = \{D_{w_1},\dotsc, D_{w_n}, D_{z_1},\dotsc, D_{z_n}\}\]
  of order $k,$ where $D_{w_j}$ and $D_{z_j}$ denote covariant differentiation along the $w_j$ and $z_j$ coordinate directions, respectively.  Then for any $j = 1,\dotsc,n$, and all $t\in[0,\rho_0]$, we claim that
 \begin{equation}\label{jacobi_exp}
 |\D^k \partial_{z_j}F(0,0,t)|_{\wt g}+ |D_t\D^k \partial_{z_j}F(0,0,t)|_{\wt g} \le C_1 e^{C_2 \rho_0},
 \end{equation}
 for some constants $C_1,C_2 > 0$ which may depend on the particular combination of derivatives which make up $\D^k$. The same estimate holds if $\partial_{z_j}F$ is replaced by $\partial_{w_j}F$, although we will not need this fact. 
 
 To prove the claim in \eqref{jacobi_exp}, we begin by noting some facts about general Jacobi fields on manifolds without conjugate points. In the notation of \cite{Bonthonneau2017}, let us fix a geodesic $\gamma$ emanating from $\wt x_0\in\wt M$ and let $\mathbb A(t)$ be the matrix Jacobi field along $\gamma$ satisfying $\mathbb A(0) = 0$ and $D_t\mathbb A(0) = I.$ Given that the tangential component of such a Jacobi field is linear in $t$, it suffices to only consider the component which acts on the orthogonal complement of $\gamma'(t)$, which we will again denote by $\mathbb A(t)$ in a slight abuse of notation. Then, since the curvature of $\wt M$ is bounded below by some $\kappa < 0$, one has that $\|\mathbb A(t)\| \le \sinh(\kappa t)$ by the Rauch Comparison Theorem (c.f. \cite[Thm 2.3]{doCarmoBook1992}). To obtain a lower bound, we appeal to \cite[Lemma 4]{Bonthonneau2017}, which shows that if $\wt M$ has no conjugate points, then for any $\ve > 0$, there exists a constant $C > 0$ such that $\|\mathbb A(t)^{-1}\| \le C$ for all $t > \ve,$ or equivalently $\|\mathbb A(t)\| \ge C^{-1}.$ Hence, for any orthogonal Jacobi vector field $J(t)$ along $\gamma$ such that $J(0) = 0$, we have that 
 \begin{equation}\label{jacobi_upper_lower}
 C^{-1}|D_t J(0)|_{\wt g}\le |J(t)|_{\wt g}\le \sinh(\kappa t)|D_t J(0)|_{\wt g}
 \end{equation}
for $t > \ve.$ Since we have assumed that $\rho_0 = d_{\wt g}(\wt x_0,\wt y_0)\ge 1$, we may make the choice of $\ve \ll 1$ independently of $\wt x_0,\wt y_0.$
 
 The next step in the proof is to observe that $\D^k\partial_{z_j}F$ satisfies an inhomogeneous Jacobi equation of the form 

\begin{equation}\label{inhom_jacobi}
D_t^2(\mathcal D^k\partial_{z_j}F) + R(\D^k\partial_{z_j}F,\partial_t F)\partial_t F + S_k = 0
\end{equation}
where $R$ is the Riemannian curvature tensor, and $S_k$ is a vector field along the variation $F$ which is induced by the pullback of a sum of tensors on $M$, evaluated on a subcollection of the vector fields $\D^{k-1}\partial_{z_j}F$, $\D^{k-1}\partial_{z_j}F$, $\partial_t F,$ where $\D^{k-1}$ is some iterated combination of elements of $\mathscr D$ of order $k-1.$ This statement is nearly identical to equation (3.17) of \cite{Blair2018} and it is proved in exactly the same way. To obtain the estimate \eqref{jacobi_exp}, we will induct on $k$. For $k = 0,$ one can use that $\partial_{z_j}F$ satisfies the homogeneous Jacobi equation and argue as in \cite{Blair2018} to see that there is a uniform constant $C_0 > 0$ so that 
\[\frac{1}{2}\partial_t\lp |\partial_{z_j}F|^2_{\wt g} + |D_t\partial_{z_j}F|_{\wt g}^2\rp \le C_0 \lp  |\partial_{z_j}F|^2_{\wt g} + |D_t\partial_{z_j}F|_{\wt g}^2\rp.\]
Since $F(w,z,0) = \varphi(w)$, it is clear that $\partial_{z_j}F$ vanishes at $t = 0,$ and hence by \eqref{jacobi_upper_lower} and Gronwall's inequality, we obtain
\begin{equation}
|\partial_{z_j}F(0,0,t)|^2_{\wt g} + |D_t\partial_{z_j}F(0,0,t)|_{\wt g}^2 \le C_1 e^{C_2 t}
\end{equation}
for some $C_1,C_2 > 0$ and for all $t\in[0,\rho_0].$ Assume now that $k \ge 1$, and set $X_t = \D^k\partial_{z_j}F(0,0,t).$ We claim that $X_t$ solves the boundary value problem
\begin{equation}
\begin{cases}
D_t^2 X_t + R(X_t,\dot\sigma_t)\dot\sigma_t + S_k = 0\\
X_0 = 0, \hskip 0.2in X_{\rho_0} = f(y_0),
\end{cases}
\end{equation}
where $\sigma_t = F(0,0,t)$ is the geodesic connecting $\wt x_0$ and $\wt y_0$, and $f$ is a vector field which is uniformly bounded. To see that $X_t$ satisfies these boundary conditions, note that 
\[F(w,z,0) = \varphi(w) \hskip 0.2in \text{and}\hskip 0.2in F(w,z,\rho_0) = \psi(z),\]
and so $X_t$ always vanishes at $t = 0$, since its definition involves applying $\partial_{z_j}$ to $F$. Furthermore, if $\D^k$ consists of any derivatives in $w$, then $X_t$ also vanishes at $t = \rho_0.$ If $\D^k$ consists only of derivatives in $z$, then $X_{\rho_0}$ is computed by repeatedly differentiating the canonical chart map $\psi,$ and is therefore uniformly bounded since $\wt M$ has bounded geometry. We then decompose $X_t = Y_t + Z_t$, where $Y_t$ satisfies the same inhomogeneous equation as $X_t$ but with $Y_0 = D_tY_0 = 0$, and $Z_t$ solves the corresponding homogeneous equation with $Z_0 = 0,$ $Z_{\rho_0} = f(y_0) - Y_{\rho_0}.$ It is shown in the proof of \cite[Lemma 3.3]{Blair2018} that $Y_t$ satisfies 
\begin{equation}\label{Y_eqn}
|Y_t|_{\wt g}+ |D_t Y_t|_{\wt g} \le C_1 e^{C_2\rho_0}
\end{equation}
for all $t\in[0,\rho_0].$ It is this step which utilizes the induction hypothesis that \eqref{jacobi_exp} holds when taking fewer than $k$ covariant derivatives of $\partial_{z_j}F$. If $f(y_0) - Y_{\rho_0} = 0$, then $Z_t$ is identically zero by the no conjugate points assumption. Otherwise, we apply \eqref{jacobi_upper_lower} to obtain that $|D_tZ_0|_{\wt g}\le |Z_t|_{\wt g}$ for all $t\in[\ve,\rho_0]$. Evaluating at $t = \rho_0$ gives $|D_t Z_0|_{\wt g}\le |f(y_0) - Y_{\rho_0}|_{\wt g},$ and so repeating the argument for the $k = 0$ case and using the boundedness of $f$ along with \eqref{Y_eqn} shows that $|Z_t|_{\wt g} + |D_t Z_t|_{\wt g}\le C_1e^{C_2 \rho_0}$ after possibly increasing $C_1,C_2.$ Thus, we have shown that 
\[|X_t|_{\wt g} + |D_tX_t|_{\wt g}\le C_1 e^{C_2\rho_0}. \]
Recalling the definition of $X_t,$ we have completed the proof of \eqref{jacobi_exp}, and therefore \propref{u_nu_bound} is proved. A similar argument holds if one replaces $\partial_{z_j}F$ by $\partial_{w_j}F$ with the boundary conditions reversed, but our result does not require it.
\end{proof}

To prove \lemref{off_diagonal}, we also required similar estimates on the inverse exponential map and squared distance function, stated below.

\begin{lem}\label{exp_map}
In the notation of \lemref{u_nu_bound}, we have
\begin{equation}\label{exp_map_bound}
|P_{\wt x}Q_{\wt y}\lp\exp_{\wt x}^{-1}(\wt y)\rp|_{\wt g} \le C_1 e^{C_2d_{\wt g}(\wt x,\wt y)}.
\end{equation}
 Here, $C_1,C_2>0$ may depend on $\nu$, $P,$ and $Q$. Moreover, we have
\begin{equation}\label{distance_bound}
|P_{\wt x} Q_{\wt y}\lp d_{\wt g}(\wt x,\wt y)^2\rp|\le C_1 e^{C_2 d_{\wt g}(\wt x,\wt y)}.
\end{equation}
\end{lem}

\begin{proof}
First let us note that \eqref{distance_bound} follows immediately from \eqref{exp_map_bound} and the fact that $\wt M$ has bounded geometry, since $d_{\wt g}(\wt x,\wt y)^2 = |\exp_{\wt x}^{-1}(\wt y)|_{\wt g}^2$. So we only need to show \eqref{exp_map_bound}. Since the metric on $\wt M$ is uniquely geodesic, the map $\exp_{\wt x}^{-1}(\wt y)$ is globally defined and $C^\infty$. We can write the action of this map as
\[(\wt x,\wt y)\mapsto (r(\wt x,\wt y),\omega(\wt x,\wt y))\in \R^+\times S^*\wt M,\]
provided that we avoid a neighborhood of the diagonal in $\wt M\times \wt M.$ We claim that the $\wt x,\wt y$ derivatives of this map are bounded exponentially in $d_{\wt g}(\wt x,\wt y)$. Furthermore, we may recall that by discussions from the proof of \propref{u_nu_bound}, it suffices to prove this in canonical coordinates. For this, we take note of the following general fact. If $G\in C^{\infty}(\R^n\times \R^n)$ and $b\in C^\infty(\R^n)$ are such that $G(a,b(a)) = 0$ and $\partial_b G(a,b(a))$ is invertible, we have that $\partial_a G(a,b(a)) + \partial_b G(a,b(a))\partial_a b(a) = 0,$ and hence
\begin{equation}\label{b_eqn}
\partial_ab(a) = -\partial_b G(a,b(a))^{-1}\partial_a G(a,b(a)).
\end{equation}
By repeated differentiation of the equation $G(a,b(a)) = 0$ with respect to $a$, one obtains that for any multiindex $\alpha,$ we can express $\partial_a^\alpha b(a)$ in terms of $\partial_b G(a,b(a))^{-1}$ times a finite linear combination of terms involving factors of $\partial_a^{\beta}\partial_b^\gamma G(a,b(a))$ for $|\beta|+|\gamma|\le |\alpha|$ and factors of the form $\partial_a^{\alpha'}b(a)$ for $|\alpha'|\le |\alpha| - 1.$ One can then use induction and \eqref{b_eqn} to show that if $|\alpha| = N$, then there exists a constant $C_N,k_N > 0$ so that 
\begin{equation}\label{b_derivs}
|\partial_a^\alpha b(a)| \!\le\! C_N\!\!\!\!\!\!\!\!\! \sum\limits_{|\beta| + |\gamma| \le N}\!\!\!\!\!\!\!\! |\partial_a^\beta\partial_b^\gamma G(a,b(a))|^{N}\!\left[ \left\|\partial_bG(a,b(a))^{-1}\right\|\!\!\lp 1 \!+\! \|\partial_b G(a,b(a))^{-1}\|^{k_N}\!\!\rp\!\!\right],
\end{equation}
where $\|\cdot\|$ here denotes the usual matrix norm.
 We now consider, in some chosen canonical coordinates $(\wt x,\wt y)$ on $\wt M$ and standard polar coodinates $(r,\omega)$ on $T_{\wt x}^*\wt M$, the function
\[G(\wt x,\wt y;r,\omega) = \exp_{\wt x}(r\omega) - \wt y.\]
So in the notation of the preceding discussion, we would have $a = (\wt x,\wt y)$ and $b(a) = (r(\wt x,\wt y),\omega(\wt x,\wt y)) = \exp_{\wt x}^{-1}(\wt y).$ By \lemref{exp_map}, we have that derivatives of $G$ are bounded exponentially in $r$. Restricted to the set where $G = 0$, we know that $r = d_{\wt g}(\wt x,\wt y)$, and hence for any $N$, there exist constants $C_N,c_N > 0$ such that
\begin{equation}\label{G_bound}
|\partial^N G|\le C_Ne^{c_Nd_{\wt g}(\wt x,\wt y)}.
\end{equation}
 Here $\partial^N$ denotes any combination of derivatives in $\wt x,\wt y,r,\omega$ with total order $N$. In what follows, we will assume that all quantities are evaluated where $r\omega = \exp_{\wt x}^{-1}(\wt y)$, unless otherwise specified. By \eqref{b_derivs} and \eqref{G_bound}, it only remains to bound the inverse matrix $\partial_{r,\omega}G^{-1}$. We achieve this by expressing it in terms of Jacobi fields between $\wt x$ and $\wt y$. In particular, $\partial_r G$ is exactly the velocity of the geodesic connecting $\wt x$ and $\wt y$, and therefore has norm 1. Also, we have that $\partial_\omega G$ is an orthogonal matrix whose columns are normal Jacobi fields $\{J_k\}_{k=2}^n$ along the geodesic connecting $\wt x$ and $\wt y$ which vanish at $\wt x$. Thus, the elements of $\partial_\omega G$ are bounded exponentially in $r$, and since the columns are orthogonal, $\partial_{r,\omega} G^T\partial_{r,\omega} G$ is a diagonal matrix $\mathbb D$ whose entries are the norms $|J_k|_g^2$ (setting $J_1 = \partial_r G$), which vanish only at $r = 0$ and are otherwise bounded away from zero \cite[Lemma 4]{Bonthonneau2017}. Thus, $\partial_{r,\omega} G^{-1} = \mathbb D^{-1}\partial_{r,\omega} G^T $ is also bounded exponentially in $r$, provided we avoid a neighborhood of $r = 0.$ Combining this with \eqref{b_derivs} and \eqref{G_bound}, the proof is complete.
\end{proof}

\section{Proof of \lemref{spectral_cluster_log}}
\label{cluster_appendix}
A key component in the proof of \thmref{logthm} with the inclusion of derivatives in $x,y$ was the spectral cluster estimate
\begin{equation}\label{spectral_cluster_d_eqn}
\sum\limits_{\lambda_j\in[\lambda,\lambda+A]}|\partial_x^\alpha\varphi_j(x)|^2 \le C_1 \lambda^{2|\alpha|}\lp A\lambda^{n-1} + Ae^{C_2/A}\max\{\lambda^{\frac{n-1}{2}},\lambda^{n-3}\}\rp
\end{equation}
for $0 < A \le 1.$ We provide a summary of the proof here, but the techniques are mostly a repetition of arguments presented in Section \ref{asymptotics}, so we do not give all the details. We begin in a manner analogous to the exposition of \cite[\S 3.2]{SoggeBook2014}. We introduce a Schwartz function $\beta\in \mathscr S(\R)$ such that $\beta \ge 0$, $\beta(0) = 1$, and $\wh\beta(t) = 0$ for $|t|\ge \frac{1}{2}\inj(M).$ This function will serve a similar role to that of $\rho$ throughout the previous sections of this article, but the key difference is the nonnegativity assumption, which is critical in what follows. Since $\beta(0) =1$, there exists some $\delta > 0$ such that $\beta(\tau) \ge \frac{1}{2}$ for $|\tau| \le \delta.$ Then, 
\[\sum\limits_{|\lambda_j-\lambda| \le A\delta} |\partial_x^\alpha\varphi_j(x)|^2  \le 2\sum\limits_{j=0}^\infty \beta\lp\frac{\lambda-\lambda_j}{A}\rp|\partial_x^\alpha \varphi_j(x)|^2,\]
where we are able to write the summation over all $j$ by the nonnegativity of $\beta$. Since $[\lambda,\lambda+A]$ can be covered by a fixed, finite number of intervals of the form ${|\lambda - \lambda_j|\le A\delta}$, we have that
\[\sum\limits_{\lambda_j\in [\lambda,\lambda + A]}|\partial_x^\alpha\varphi_j(x)|^2 \le C\sum\limits_{j=0}^\infty \beta\lp\frac{\lambda - \lambda_j}{A}\rp|\partial_x^\alpha\varphi_j(x)|^2\]
for some constant $C > 0.$
By Fourier inversion, we have
\begin{align*}
\beta\lp\frac{\lambda - \lambda_j}{A}\rp  = \frac{1}{2\pi}\int\limits_{-\infty}^\infty A\wh\beta(At)e^{it(\lambda-\lambda_j)}\,dt  = \frac{1}{\pi}\int\limits_{-\infty}^\infty A\wh\beta(At)e^{it\lambda}\cos(t\lambda_j)\,dt - \beta\lp\frac{\lambda + \lambda_j}{A}\rp.
\end{align*}
Since $\beta$ is Schwartz, we have an estimate of the form 
\[\left|\beta\lp\frac{\lambda+\lambda_j}{A}\rp\right|\le C_N(1 + A^{-1}|\lambda +\lambda_j|)^{-N}\]
 for any $N$. Recalling that $A^{-1} \ge 1$ and $\lambda_j \ge 0$ for all $j$, we have that 
\[\sum\limits_{\lambda_j\in [\lambda,\lambda + A]}|\partial_x^\alpha\varphi_j(x)|^2 \le \frac{1}{\pi}\left|\int\limits_{-\infty}^\infty A\wh\beta(At)e^{it\lambda}\partial_x^\alpha\partial_y^\alpha K(t,x,y)\big|_{x = y}\,dt\right| + \mathcal O\lp \lambda^{-N}\rp,\]
for any $N$ as $\lambda \to \infty,$ where the implicit constant in the $\mathcal O(\lambda^{-N})$ term is independent of $A\in (0,1]$. By \propref{parametrix_error}, the proof of \eqref{spectral_cluster_d_eqn} can be reduced to showing that
\begin{align*}
&\frac{1}{\pi}\left|\int\limits_{-\infty}^\infty A\wh\beta(At)e^{it\lambda}\partial_x^\alpha\partial_y^\alpha K_N(t,x,y)\big|_{x= y}\,dt \right| \le C_1\lambda^{2|\alpha|}\lp A\lambda^{n-1} + Ae^{C_2/A}\max\{\lambda^{\frac{n-1}{2}},\lambda^{n-3}\} \rp,
\end{align*}
where $K_N(t,x,y)$ is the $N$ partial sum of the Hadamard parametrix, defined by \eqref{parametrix_partial_sums}.
This is proved by repeating the arguments from Section \ref{asymptotics} with $\frac{\sin t\lambda}{t}$ replaced by $e^{it\lambda}$, yielding an integrand which is one degree less singular in $t$, which then produces one lower power of $\lambda$ in the asymptotic expansion. In particular, by the proof of \lemref{leading_term_mu}, we have
\begin{equation}\label{cluster_leading_term}
\left|\partial_x^\alpha\partial_y^\alpha\lp u_0(x,y)\!\!\int\limits_{-\infty}^\infty \!\!\! A\wh\beta(At)e^{it\lambda}\partial_t W_0(t,d_g(x,y))\,dt\rp \bigg|_{x=y}\right| \le C A \lambda^{n-1 + 2|\alpha|}.
\end{equation}
For $\nu \ge 1$, we can repeat the proof of \lemref{near_diagonal_terms} to obtain
\begin{equation}\label{cluster_on_diag}
\left|\partial_x^\alpha\partial_y^\alpha\!\lp\!\! u_\nu(x,y)\!\! \int\limits_{-\infty}^\infty\!\!\! A\wh\beta(At)e^{it\lambda} \partial_tW_\nu(t,d_g(x,y))\rp \bigg|_{x=y} \right| \le C_\nu\max\{\lambda^{n-2\nu + 2|\alpha|},e^{C/A}\}.
\end{equation}
 That the exponent here is $n-2\nu + 2|\alpha|$ rather than $n-\nu-1 +2|\alpha|$ is due to the fact that in the integration by parts used to obtain \eqref{near_diag_ibp}, we only obtain the term where $\ell = 0$, since $\exp_x^{-1}(x) = 0$. Also, recall that in the proof of \lemref{near_diagonal_terms}, the $e^{C/A}$ term yielded a factor of $\lambda^{\delta}$ for some small $\delta > 0$, but this was due to the fact that we chose $A = \frac{1}{c\log\lambda}$. Since we have stated the lema for arbitrary $A$, we leave the above as is. Finally, for the terms arising from the non-identity elements of the deck transformation group, we have
\begin{align}\label{cluster_off_diag}
\begin{split}
&\left|P_{\wt x}Q_{\wt y}\lp u_\nu(\wt x,\gamma \wt y)\int\limits_{-\infty}^\infty A \wh\beta(At)e^{it\lambda} \partial_t W_\nu(t,d_{\wt g}(\wt x,\gamma\wt y))\rp\bigg|_{\wt x=\wt y}\right| \le C_\nu e^{C/A}\max\{\lambda^{\frac{n-1}{2} - \nu +2|\alpha| },1\}
\end{split}
\end{align}
for any $P,Q\in \textnormal{Diff}(\wt M)$ of orders $|\alpha|$ and $|\beta|,$ respectively, by the arguments in the proof of \lemref{off_diagonal}. Combining these estimates with the fact that there are at most $\mathcal O\lp e^{C/A}\rp$ deck transformations $\gamma$ for which the corresponding term is nonzero, we thus obtain \eqref{spectral_cluster_d_eqn}.

\bibliography{keeler_bibliography}{}
\bibliographystyle{plain}


\end{document}